\documentclass{article}

\usepackage{a4wide}

\usepackage{amsmath, amsthm, amssymb}
\usepackage{stmaryrd}
\usepackage[latin1]{inputenc}
\usepackage{subfigure}
\usepackage{enumerate}
\usepackage{cancel}
\usepackage{color}
\usepackage{float}

\usepackage{wrapfig}

\newcommand{\nbanac}[1]{b_{#1}}
\newcommand{\Catalan}[1]{C_{#1}}
\newcommand{\Bessel}[2]{J_{#1}(#2)}

\newcommand{\Exponential}[1]{exp\left({#1}\right)}
\newcommand{\hookformula}[1]{NA \left( {#1} \right)}
\newcommand{\pattern}{\begin{array}{c}\includegraphics[scale=1.0]{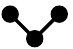}\end{array}}
\newcommand{\tmppp}{parallelogram polyomino}
\newcommand{\pp}{\tmppp }
\newcommand{\pps}{\tmppp es }

\def\leftedgecolor{red}
\def\rightedgecolor{blue}

\def\T{{\mathcal{T}}}
\def\funT{\mathsf{T}}

\def\myS{S}

\def\myn{n_0}

\newcommand\pref[1]{(\ref{#1})}
\def\ie{{\em i.e. }}

\usepackage{graphicx}
\title{Combinatorics of non-ambiguous trees}
\author{Jean-Christophe Aval, Adrien Boussicault, Mathilde Bouvel, Matteo Silimbani \footnote{LaBRI - CNRS, Universit\'e de Bordeaux, 351 Cours de la Lib\'eration, 33405 Talence, France. All authors are supported by ANR -- PSYCO project (ANR-11-JS02-001).}}

\date{~}

\newtheorem{theorem}{Theorem}
\newtheorem{proposition}[theorem]{Proposition}
\newtheorem{lemma}[theorem]{Lemma}
\newtheorem{corollary}[theorem]{Corollary}
\newtheorem{remark}[theorem]{Remark}

\begin{document}
\maketitle
\begin{abstract}
This article investigates combinatorial properties of non-ambiguous trees.
These objects we define may be seen either as binary trees drawn on a grid with some constraints,
or as a subset of the tree-like tableaux previously defined by Aval, Boussicault and Nadeau.
The enumeration of non-ambiguous trees satisfying some additional constraints allows us
to give elegant combinatorial proofs of identities due to Carlitz, and to Ehrenborg and Steingr{\'{\i}}msson.
We also provide a hook formula to count the number of non-ambiguous trees with a given underlying tree.
Finally, we use non-ambiguous trees to describe a very natural bijection between parallelogram polyominoes and binary trees.
\end{abstract}

\section{Introduction}

It is well known that the Catalan numbers $\Catalan{n}=\frac{1}{n+1}{2n \choose n}$ enumerate many combinatorial objects, such as binary trees and parallelogram polyominoes. Several bijective proofs in the literature show that \pps are enumerated by Catalan numbers, the two most classical being Delest-Viennot's bijection with Dyck paths~\cite{delvie} and Viennot's bijection with bicolored Motzkin paths~\cite{delvie}.

In this paper we demonstrate a bijection -- which we believe is more natural -- between binary trees and parallelogram polyominoes. In some sense, we show that \pps may be seen as two-dimensional drawings of binary trees. This point of view gives rise to a new family of objects -- we call them {\em non-ambiguous trees} -- which are particular compact embeddings of binary trees in a grid.

The tree structure of these objects leads to a hook formula for the number of non-ambiguous trees with a given underlying tree.
Unlike the classical hook formula for trees due to Knuth (see~\cite{knuth}, $\S 5.1.4$, Exercise 20),
this one is defined on the edges of the tree.

Non-ambiguous trees are in bijection with
permutations such that all their (strict) excedances stand at the beginning of the permutation word.
Ehrenborg and Steingr{\'{\i}}msson in~\cite{Ehrenborg2000284} give a closed formula
(involving Stirling numbers of the second kind) for the number of such permutations.
We show that this formula  can be easily proved using non-ambiguous trees
and a variation of the insertion algorithm for tree-like tableaux introduced in~\cite{abn}.
Indeed, non-ambiguous trees can also be seen as a subclass of tree-like tableaux,
objects defined in~\cite{abn} and that are in bijection with permutation tableaux~\cite{steiwil07} or alternative tableaux~\cite{nad11,vie07}.

A particular subclass of non-ambiguous trees leads to unexpected combinatorial interpretations.
We study complete non-ambiguous trees, defined as non-ambiguous trees such that their
underlying binary tree is complete,
and show that their enumerating sequence is related to the formal power series of the logarithm of the Bessel function of order $0$.
This gives rise to new combinatorial interpretations of some identities due to Carlitz~\cite{car},
and to the proof of a related identity involving Catalan numbers, which had been conjectured by P. Hanna.

The paper is organized as follows: in Section~\ref{definitions} we define non-ambiguous trees.
Then, in Section~\ref{sec:enumeration} we give the enumeration of non-ambiguous trees satisfying certain constraints:
those contained in a given rectangular box, and those with a fixed underlying tree.
Section~\ref{enumeration_anac} introduces the family of complete non-ambiguous trees,
studies the relations between this family and the Bessel function,
and proves combinatorial identities involving Catalan numbers and the sequence enumerating complete non-ambiguous trees.
In Section~\ref{polyominoes} we describe our new bijection between binary trees and parallelogram polyominoes.
To conclude, we present in Section~\ref{sec:open} some perspectives related to our study.

\section{Definitions and notations}
\label{definitions}

In this paper, trees are embedded in a bidimensional grid $\mathbb{N}\times\mathbb{N}$.
The grid is not oriented as usual:
the $x$-axis has south-west orientation, and the $y$-axis has south-east orientation, as shown on Figure~\ref{grid}.

\begin{figure}[H]
	\begin{minipage}[b]{.46\linewidth}
		$$
		\begin{array}{c}
		\includegraphics[scale=1.0]{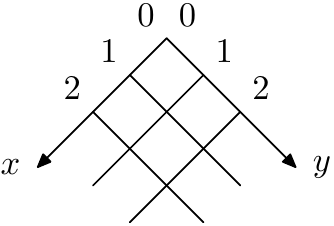}
		\end{array}
		$$
		\vspace{-.6cm}
		\caption{The underlying grid for non-ambiguous trees\label{grid}}
	\end{minipage} \hfill
	\begin{minipage}[b]{.46\linewidth}
		$$
		\begin{array}{c}
		\includegraphics[scale=1.0]{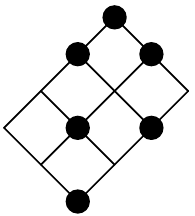}
		\end{array}
		\longleftrightarrow
		\begin{array}{c}
		\includegraphics[scale=1.0]{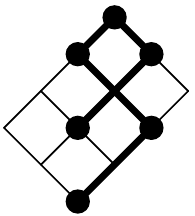}
		\end{array}
		$$
		\vspace{-.6cm}
		\caption{The edges of a non-ambiguous tree are not necessary\label{notnecesary}}
	\end{minipage}
\end{figure}

Every $x$-oriented (resp. $y$-oriented) line will be called column (resp. row).
Each column (resp. row) on this grid is numbered with an integer corresponding to its $y$ (resp. $x$) coordinate. A vertex $v$ located on the intersection of two lines has the coordinate representation: $(X(v),Y(v))$.

A non-ambiguous tree may be seen as a binary tree embedded in the grid in such a way that its edges may be recovered from the embedding of its vertices in the grid (Figure~\ref{notnecesary}). In other words, the vertices determine the tree without ambiguity, whence the name of these objects.

Formally, a \emph{non-ambiguous tree} of size $n$ is a set $A$ of $n$ points $(x,y)\in\mathbb{N}\times\mathbb{N}$ such that:
\begin{enumerate}
\item \label{condition_1_ana} $(0,0)\in A$; we call this point the \emph{root} of $A$;
\item \label{condition_2_ana} given a non-root point $p\in A$, there exists one point $q\in A$ such that $Y(q)<Y(p)$ and $X(q)=X(p)$, or one point $r\in A$ such that $X(r)<X(p)$, $Y(r)=Y(p)$, but not both (which means that the pattern $\pattern$ is avoided);
\item \label{condition_3_ana} there is no empty line between two given points: if there exists a point $p\in A$ such that $X(p)=x$ (resp. $Y(p)=y$), then for every $x'<x$ (resp. $y'<y$) there exists $q\in A$ such that $X(q)=x'$ (resp. $Y(q)=y'$).
\end{enumerate}

Figure~\ref{analist} shows some examples and counterexamples of non-ambiguous trees.

\begin{figure}[H]
 \centering
 \subfigure[Four non-ambiguous trees]
   {$\begin{array}{c}
\includegraphics[scale=0.8]{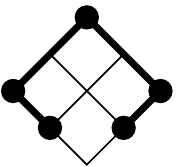}
\end{array}
\begin{array}{c}
\includegraphics[scale=0.8]{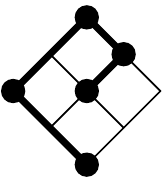}
\end{array}
\begin{array}{c}
\includegraphics[scale=0.8]{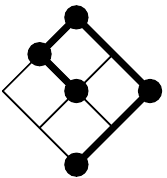}
\end{array}
\begin{array}{c}
\includegraphics[scale=0.8]{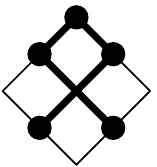}
\end{array}$
}
 \subfigure[These four are not non-ambiguous trees]
   {$\begin{array}{c}
\includegraphics[scale=0.8]{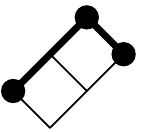}
\end{array}
\begin{array}{c}
\includegraphics[scale=0.8]{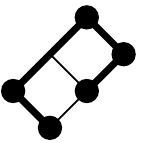}
\end{array}
\begin{array}{c}
\includegraphics[scale=0.8]{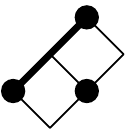}
\end{array}
\begin{array}{c}
\includegraphics[scale=0.8]{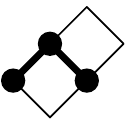}
\end{array}$}
\vspace{-.4cm}
 \caption{Some examples and counterexamples of non-ambiguous trees} \label{analist}
 \end{figure}


It is straightforward that a non-ambiguous tree $A$ has a tree structure: except for the root, every point $p\in A$ has a unique parent, which is the
nearest point $q$ preceding $p$ in the same row (resp. column). In this case, we will say that $p$ is the \emph{right child} (resp. \emph{left child}) of $q$.
In this paper, we orient every edge of a tree from the root to the leaves.
We shall denote by $\funT(A)$ the underlying binary tree associated to $A$.


Figure~\ref{ana_n_4} shows all the non-ambiguous trees of size $4$, grouping inside a rectangle those having the same underlying binary tree.

\begin{figure}[H]
$$
\begin{array}{c}
\includegraphics[scale=1.0]{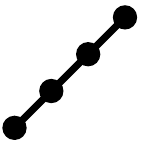}
\end{array}
\begin{array}{c}
\includegraphics[scale=1.0]{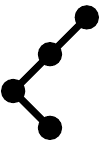}
\end{array}
\begin{array}{c}
\includegraphics[scale=1.0]{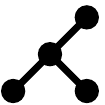}
\end{array}
\begin{array}{c}
\includegraphics[scale=1.0]{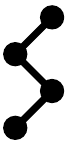}
\end{array}
\begin{array}{c}
\includegraphics[scale=1.0]{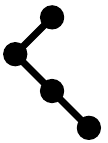}
\end{array}
\begin{array}{c}
\includegraphics[scale=1.0]{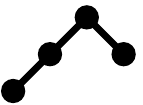}
\end{array}
\boxed{
\begin{array}{c}
\includegraphics[scale=1.0]{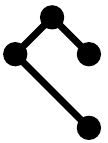}
\end{array}
\begin{array}{c}
\includegraphics[scale=1.0]{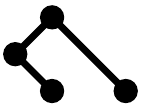}
\end{array}
}
$$
$$
\begin{array}{c}
\includegraphics[scale=1.0]{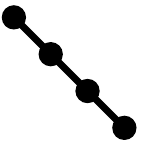}
\end{array}
\begin{array}{c}
\includegraphics[scale=1.0]{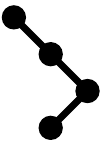}
\end{array}
\begin{array}{c}
\includegraphics[scale=1.0]{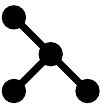}
\end{array}
\begin{array}{c}
\includegraphics[scale=1.0]{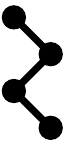}
\end{array}
\begin{array}{c}
\includegraphics[scale=1.0]{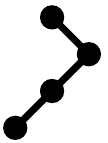}
\end{array}
\begin{array}{c}
\includegraphics[scale=1.0]{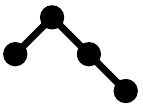}
\end{array}
\boxed{
\begin{array}{c}
\includegraphics[scale=1.0]{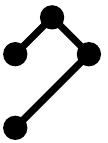}
\end{array}
\begin{array}{c}
\includegraphics[scale=1.0]{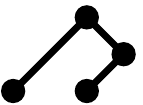}
\end{array}
}
$$
\vspace{-.5cm}
\caption{The $16$ non-ambiguous trees of size $4$}\label{ana_n_4}
\end{figure}

\begin{remark}\label{remark:tlt}
A tree-like tableau~\cite{abn} of size $n$ is a set of $n$ points placed in the boxes of a Ferrers diagram
such that conditions 1, 2, 3 defining non-ambiguous trees are satisfied.
Figure~\ref{example:tlt} shows an example of a tree-like tableau of size $7$.
It should be clear that non-ambiguous trees are in bijection with tree-like tableaux with rectangular shape.
\end{remark}

\begin{figure}[ht]
$$
\begin{array}{c}
\includegraphics[scale=0.8]{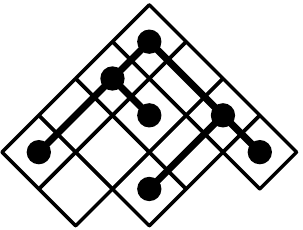}
\end{array}
$$
\caption{A tree-like tableau}\label{example:tlt}
\end{figure}

%
%
%
%
%
%
%
\section{Enumeration of non-ambiguous trees}
\label{sec:enumeration}

Eventhough non-ambiguous trees are new combinatorial objects, their enumeration has been previously studied. Indeed,
non-ambiguous trees of size $n$ are in bijection with permutations of size $n$ with all their strict excedances at the beginning,
whose enumeration has been studied by Ehrenborg and Steingr{\'{\i}}msson in~\cite{Ehrenborg2000284}.
The size preserving bijection (that we do not detail here) between these two families of combinatorial objects
is a consequence of Lemma $5$ in~\cite{steiwil07} and of results proved in~\cite{avbona13}.
The sequence $(a_n)_{n \ge 1}$ counting the number of non-ambiguous tree of size $n$ is referenced in~\cite{oeis} as ${\tt A136127} = [1, 2, 5, 16, 63, 294, 1585, 9692,\dots]$, but no simple formula is known.
In this section, we provide enumerative formulas for non-ambiguous trees with additional constraints.

\subsection{Non-ambiguous trees inside a fixed rectangle}
\label{subsec:fixed-box}

In this section, we make extensive use of Remark~\ref{remark:tlt} and we view non-ambiguous trees as tree-like tableaux of rectangular shape.
For brevity, we write TLT for tree-like tableau in the sequel.
Of particular importance for our purpose is the insertion procedure for TLTs, which gives to these objects
a recursive structure.
Since it is not the central purpose of the present paper, we shall not recall here all the definitions
and properties of TLTs, and we refer the reader to~\cite{abn}.
Figure~\ref{example:tlt} shows an example of a TLT of size 7.
In full generality, the size of TLT is given by the numbers of dotted cells.

Moreover, in this section, we use a different drawing convention for non-ambiguous trees, inherited from the drawing convention for TLTs:
we will draw non-ambiguous trees with the $x$-axis vertical, and the $y$-axis horizontal.

Given a non-ambiguous tree, its $x$-size (resp. $y$-size) may be defined
as the maximum of the $x$-coordinate (resp. $y$-coordinate) of its points.
The aim of this subsection is to give a formula for the number $A(k,\ell)$
of non-ambiguous trees with $x$-size equal to $\ell$ and $y$-size equal to $k$.
Because the size of a TLT is given by its semi-perimeter$-1$,
remark that the size $n$ of such a non-ambiguous tree is given by $n = k + \ell -1$.

We denote by $c(n,j)$ the unsigned Stirling numbers of the first kind,
\ie the number of permutations of size $n$ with exactly $j$ disjoint cycles.

We shall prove that for every integers $n,\ell$, one has:
\begin{equation}
\label{eq:fixed-box}
\sum_{k=1}^n c(n,k)\, A(k,\ell) = n^{\ell-1} \, n! \,.
\end{equation}

Inverting Equation~\pref{eq:fixed-box}, we obtain that it is equivalent to Equation~\pref{eq:ES} below:
\begin{equation}
\label{eq:ES}
A(k,\ell) = \sum_{i=1}^{k} {(-1)}^{k-i}\, S(k,i)\, i!\, i^{\ell -1},
\end{equation}
where $S(k,i)$ denotes the Stirling numbers of the second kind, \ie the number
of partitions of a set of  $k$ elements into $i$ non-empty parts.
It is known~\cite{steiwil07,avbona13} that $A(k,\ell)$ is equal to the number of permutations of size $k+\ell-1$
with exactly $k$ strict excedances in position $1,2,\dots,k$.
Consequently, Corollary 6.6 in~\cite{Ehrenborg2000284} provides a proof of Equation~\pref{eq:fixed-box}.
More precisely, in~\cite{Ehrenborg2000284}, Ehrenborg and Steingr{\'{\i}}msson prove~\pref{eq:ES} by an inclusion-exclusion
argument, and then deduce~\pref{eq:fixed-box} by inversion.

Our goal is to prove Equation~\pref{eq:fixed-box} directly and combinatorially.
For this purpose, we provide combinatorial interpretations of the numbers $c(n,k)$ and $A(k,\ell)$ that appear in Equation~\pref{eq:fixed-box} with
TLTs and non-ambiguous trees respectively. As mentioned in Remark~\ref{remark:tlt}, non-ambiguous trees are nothing but TLTs with a rectangular shape, so that
our interpretation of $c(n,k)$ and $A(k,\ell)$ is then with unified objects. This is the key in proving Equation~\pref{eq:fixed-box} by a combinatorial approach.

Recall that by definition $A(k,\ell)$ is the number of non-ambiguous trees with $x$-size equal to $\ell$ and $y$-size equal to $k$. On the other hand,
we claim that $c(n,j)$ counts the number of TLTs of size $n$ with exactly $j$ points in their first row.
Indeed, this follows by definition of $c(n,j)$ and the fact that TLTs of size $n$ with exactly $j$ points in their first row are in bijection with permutations of size $n$ with exactly $j$ disjoint cycles.
This is a consequence of Theorem $4.2$ in~\cite{Bur07},
re-formulated in terms of TLTs as in~\cite{avbona13}.

\begin{proposition}
\label{prop:ES-TLT}
For every integers $n,\ell$, we have the following identity:
\begin{equation}
\label{eq:ES-TLT}
\sum_{k=1}^n c(n,k)\, A(k,\ell) = n^{\ell-1} \, n! \,
\end{equation}
where $c(n,k)$ denotes the number of TLTs of size $n$ with exactly $k$ points in their first row.
\end{proposition}

The proof consists in two steps:
\begin{itemize}
\item introduce a set $\T_{n,\ell}$ whose cardinality will be proved to be $n^{\ell-1} \, n!$ (Lemma~\ref{lem:m-insert});
\item prove that the elements of $\T_{n,\ell}$ are in bijection  with pairs $(b,a)$ enumerated
by the left-hand side of~(\ref{eq:ES-TLT}) (Lemma~\ref{lem:m-cut}).
\end{itemize}

We define $\T_{n,\ell}$ as the set of TLTs 
of size $n+\ell-1$ whose first $\ell$ rows are of equal length.
Figure~\ref{fig:TLT-l} shows an element of $\T_{5,2}$.

\begin{figure}[H]
$$
\begin{array}{c}
\includegraphics[scale=0.8]{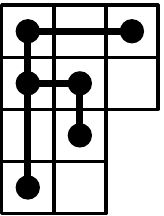}
\end{array}
$$
\caption{A tree-like tableau of $T_{5,2}$\label{fig:TLT-l}}
\end{figure}

To compute the cardinality of $\T_{n,\ell}$, we will use a slight variation of the insertion procedure defined in~\cite{abn}.
This new insertion procedure depends on $\ell$, and we shall call it the {\em $\ell$-insertion procedure}.

\paragraph*{The $\ell$-insertion procedure.}
Let $T$ be an element of $\T_{n,\ell}$. 
The $\ell$-special box of $T$ is defined as
\begin{itemize}
\item the right-most box of the $\ell$-th row, if the $(\ell+1)$-st row is strictly smaller;
\item the right-most dotted box at the bottom of a column, otherwise.
\end{itemize}
Given an integer $m$ in $\{1,\dots,n\}$, we associate to it an edge $e_m$ of the South-East border of $T$
by labeling these edges from South-West to North-East.
Note that we exclude the (vertical) edges $e_{n+1},\dots,e_{n+\ell}$ at the right of the first $\ell$ rows of $T$.
As in the insertion procedure of~\cite{abn}, if $e_m$ is a horizontal (resp. vertical) edge of the border of $T$,
we add a row (resp. column) to $T$ below (resp. to the right of) $e_m$,
composed of empty boxes except for one dotted box: the one below (resp. to the right of) $e_m$.
Next, if the dotted box added is to the left of the $\ell$-special box of $T$, we add a {\em ribbon}
(a connected  set of empty boxes without any $2\times 2$ square) adjacent simultaneously to the $\ell$-special box of $T$
and to the dotted box added.

\smallskip

Figure~\ref{fig:m-insert} illustrates the $\ell$-insertion procedure.

\begin{figure}[H]
\begin{center}
  \includegraphics[width=0.6\textwidth]{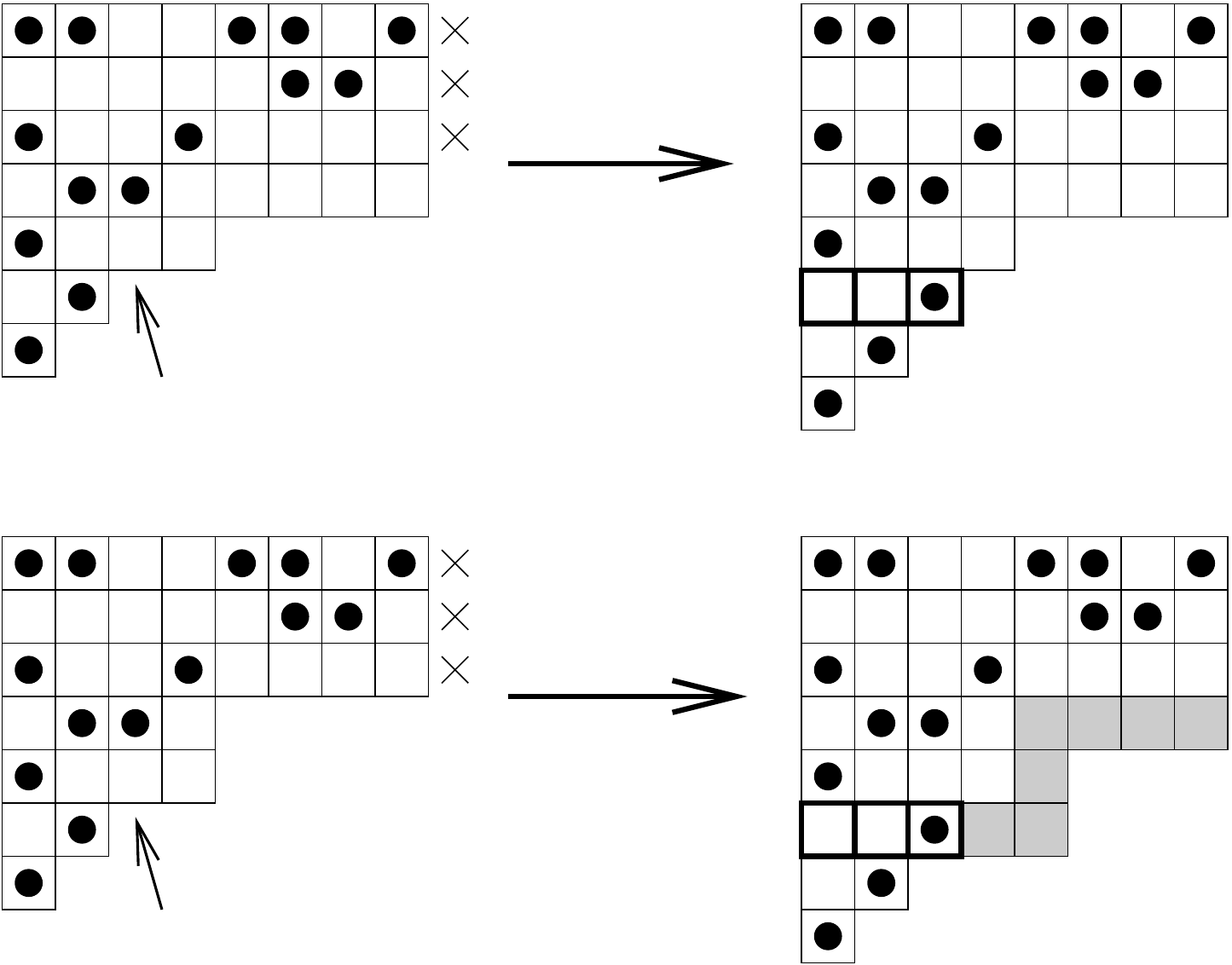}
\caption{The $\ell$-insertion procedure on two examples for $\ell=3$ and $m=5$. \label{fig:m-insert}}
\end{center}
\end{figure}

\begin{lemma}
\label{lem:m-insert}
The cardinality of $\T_{n,\ell}$ is given by
\begin{equation}
\label{eq:m-insert}
\#\T_{n,\ell} = n^{\ell-1} \, n!\,.
\end{equation}
\end{lemma}

\begin{proof}
We prove the above by induction on $\ell$.
By \cite[Theorem 2.2]{abn}, the number of unrestricted TLTs of size $n$ is equal to $n!$, thus (\ref{eq:m-insert}) is true for $\ell=1$.

For the inductive step, we claim that the $\ell$-insertion procedure gives a bijection between $\T_{n,\ell}\times\{1,\dots,n\}$ and $\T_{n,\ell+1}$.
To prove this fact, we first notice that the result of the $\ell$-insertion procedure on $T \in \T_{n,\ell}$ and $e_m$ with $m \in \{1,\dots,n\}$ is a TLT 
whose first $\ell+1$ rows are of equal length,
thus an element of $\T_{n,\ell+1}$.
Next, as explained in~\cite{abn}, we observe that the new box added by $\ell$-insertion in an element of $\T_{n,\ell+1}$
is easy to recognize: it is the rightmost among dotted boxes at the bottom of a column.
Thus we are able to invert the $\ell$-insertion procedure.
This ends the proof of the lemma.
\end{proof}

Now we will send bijectively an element of $\T_{n,\ell}$ on pairs of objects $(b,a)$
enumerated by the left-hand side of~(\ref{eq:ES-TLT}).
The bijection relies on the {\em $\ell$-cut procedure} described below, and illustrated in Figure~\ref{fig:m-cut}.

\paragraph*{The $\ell$-cut procedure.}
Let $T$ be an element of $\T_{n,\ell}$.
We first cut $T$ by putting the first $\ell$ rows in $a'$
and the next rows in $b'$. 
Now we see $b'$ as part of a TLT.
We add to it a first row
whose length equals the width of $a'$.
We observe
that there is exactly one way to put dots in this row to obtain a TLT, that we denote $b$:
we are forced to put  dots in the boxes corresponding
to non-empty columns in $a'$.
Next we remove empty columns in $a'$ to get
a  non-ambiguous tree with $\ell$ rows, denoted $a$.

\begin{figure}[H]
\begin{center}
	$
	\begin{array}{c}
		\includegraphics[width=0.18\textwidth]{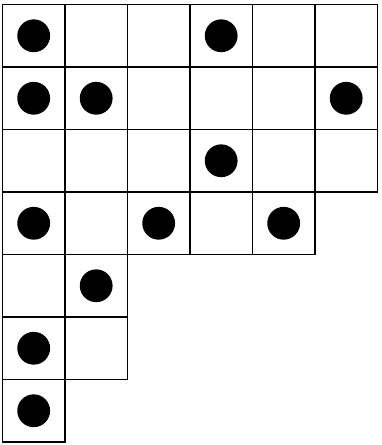}
	\end{array}
	\longrightarrow
	\begin{array}{c}
		\includegraphics[width=0.18\textwidth]{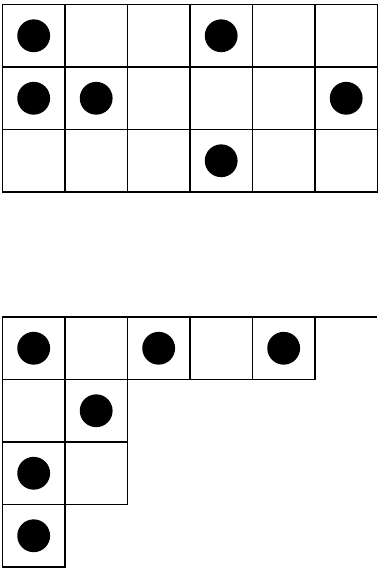}
	\end{array}
	\longrightarrow
	\begin{array}{c}
		\includegraphics[width=0.18\textwidth]{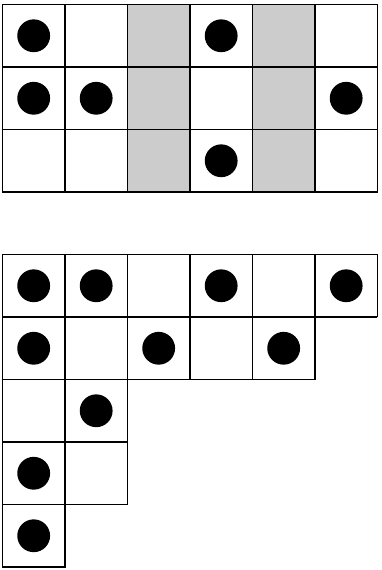}
	\end{array}
	\longrightarrow
	\begin{array}{c}
  		\includegraphics[width=0.18\textwidth]{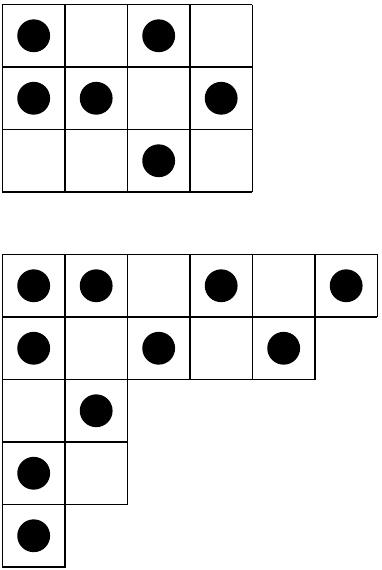}
	\end{array}
	$
\end{center}
\caption{The $\ell$-cut procedure for $\ell=3$. \label{fig:m-cut}} 
\end{figure}

\begin{lemma}
\label{lem:m-cut}
Elements of $\T_{n,\ell}$ are in bijection with pairs $(b,a)$ such that
\begin{enumerate}
\item $b$ is a TLT of size $n$,
\item $a$ is a non-ambiguous tree with $\ell$ rows,
\item the number of dotted boxes in the first row of $b$ and the width of  $a$ are equal.
\end{enumerate}
\end{lemma}

\begin{proof}
It should be clear the $\ell$-cut procedure is invertible,
which implies the lemma.
\end{proof}

Putting together Lemmas~\ref{lem:m-insert} and~\ref{lem:m-cut}
gives a bijective proof of Proposition~\ref{prop:ES-TLT}.

\subsection{Non-ambiguous trees with a fixed underlying tree: a new hook formula}
\label{subsec:hookformula}

Let $T$ be a binary tree.
We define $NA(T)$ as the number of non-ambiguous trees $A$ such that
their underlying binary tree $\funT(A)$ is $T$.
The aim of this section is to get a formula for $NA(T)$: this will be done by Proposition~\ref{prop:hook},
which shows that $NA(T)$ may be expressed by a new and elegant hook formula on the edges of $T$.
To do this, we encode any non-ambiguous tree $A$ by a triple $\Phi(A)=(T,\alpha_L,\alpha_R)$
where $T$ is a binary tree, and $\alpha_L$ (resp. $\alpha_R$) is a word called the left (resp. right) code of $A$.
To distinguish the vertices of $A$, we label them by integers from $1$ to the size of $A$, as shown on Figure~\ref{fig:ex_ANA}.

\begin{figure}[H]
$$
A =
\begin{array}{c}
\includegraphics[scale=1.0]{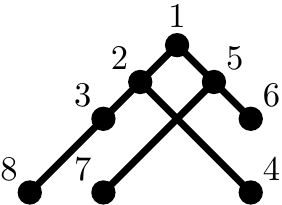}
\end{array}\ ,  \qquad T =
\begin{array}{c}
\includegraphics[scale=1.0]{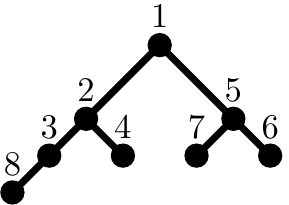}
\end{array}\ .
$$
\vspace{-.5cm}
\caption{A non-ambiguous tree $A$ with labeled vertices, and the associated binary tree $T$}\label{fig:ex_ANA}
\end{figure}

The first entry in $\Phi(A)$ is the underlying binary tree $T$ associated to $A$.
Observe that we keep the labels on vertices when we extract the underlying binary tree.
Now we denote by $V_L$ (resp. $V_R$) the set of the end points of the left (resp. right) edges of $A$, which gives $V_L=\{2,3,7,8\}$ and $V_R=\{4,5,6\}$ on the example in Figure~\ref{fig:ex_ANA}.
The definition of non-ambiguous trees ensures that the set $\{X(v), v\in V_L\}$ is the interval $\{1,\dots,\vert V_L\vert\}$.
Thus for $i=1,\dots, \vert V_L\vert $, we may set $\alpha_L(i)$ as the unique label $v\in V_L$ such that $X(v)=i$,
and we proceed symmetrically for $\alpha_R$.
On the example of Figure~\ref{fig:ex_ANA}, we have: $\alpha_L= 2378$ and $\alpha_R=564$.
Our starting point is the following lemma.

\begin{lemma}\label{Phi}
The application $\Phi$ which sends $A$ to the triple $(T,\alpha_L,\alpha_R)$ is injective.
\end{lemma}

\begin{proof}
Consider a non-ambiguous tree $A$ with $n$ vertices and $\Phi(A)=(T,\alpha_L,\alpha_R)$. From the definition of non-ambiguous tree, we know that the $X$ and $Y$-coordinates of the root $r$ are $X(r)=Y(r)=0$. Moreover, for every left edge $(s,t)$ of $T$ we have
$$Y(s)=Y(t)\textrm{ and }\alpha_L(X(t))=t,$$
namely, $X(t)$ is the position where $t$ appears in $\alpha_L$.
Similarly, for every right edge $(s,t)$, we have
$$X(s)=X(t)\textrm{ and }\alpha_R(Y(t))=t,$$
namely, $Y(t)$ is the position where $t$ appears in $\alpha_R$.
It is easy to check that we have $2n$ independent equations in $2n$ variables (the $X$ and $Y$-coordinates of the vertices), so we get a unique non-ambiguous tree.
\end{proof}

Lemma~\ref{Phi} allows us to encode a non-ambiguous tree $A$ by a triple $(T,\alpha_L,\alpha_R)$,
where $T$ is a binary tree, and $\alpha_L$ (resp. $\alpha_R)$ is a word in which every label $v\in V_L$
(resp. $V_R$) appears exactly once.
Of course, $\Phi$ is not surjective on such triples. If we take $T=\begin{array}{c}
\includegraphics[scale=0.5]{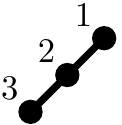}
\end{array}$, it should be clear that $\alpha_L$ is forced to be $23$.
Consequently, our next task is to characterize the pairs of codes $(\alpha_L,\alpha_R)$
which are compatible with a given binary tree $T$, \ie such that $(T,\alpha_L,\alpha_R)$
is in the image of $\Phi$.
In order to describe this characterization, we need to define  partial orders on the sets $V_L$ and $V_R$.
The pairs $(\alpha_L,\alpha_R)$ of compatible codes will be seen to correspond to pairs of linear extensions
of the posets $V_L$ and $V_R$.
The posets are defined as follows: given $a,b\in V_L$ (resp. $V_R$), we say that $a\leq b$ if and only if there exists a path in the oriented tree starting from $a$ and ending at $b$.
Figure~\ref{ex:posets} 
(with minima at the top) illustrates this notion.
\begin{figure}[H]
$$
T =
\begin{array}{c}
\includegraphics[scale=1.0]{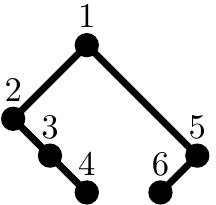}
\end{array}
\hspace{1cm}
V_L =
\begin{array}{c}
\includegraphics[scale=1.0]{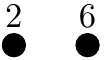}
\end{array}
\hspace{1cm}
V_R =
\begin{array}{c}
\includegraphics[scale=1.0]{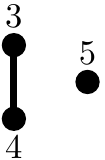}
\end{array}
$$
\vspace{-.5cm}
\caption{
The posets $V_L$ and $V_R$ of a tree $T$
\label{ex:posets}
}
\end{figure}

The next lemma is the crucial step to prove Proposition~\ref{prop:hook}.
\begin{lemma}\label{lemma:crux}
Given a binary tree $T$, the pairs of codes  compatible with $T$
are exactly the pairs $(\alpha_L,\alpha_R)$ where $\alpha_L$ is a linear extension of $V_L$
and $\alpha_R$ is a linear extension of $V_R$.

\noindent Moreover, such pairs $(\alpha_L,\alpha_R)$ are in bijection with non-ambiguous trees with underlying tree $T$.
\end{lemma}

Figure~\ref{anac_from_posets} gives these compatible codes, together with the corresponding
non-ambiguous trees, in the case of the tree $T$ of Figure~\ref{ex:posets}.
\begin{figure}[H]
$$
\begin{array}{c}
(26,534)\\
\includegraphics[scale=1.0]{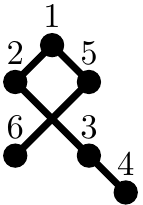}
\end{array}
\begin{array}{c}
(26,354)\\
\includegraphics[scale=1.0]{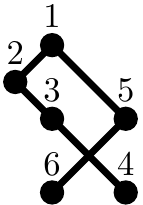}
\end{array}
\begin{array}{c}
(26,345)\\
\includegraphics[scale=1.0]{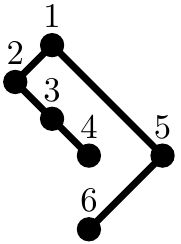}
\end{array}
\begin{array}{c}
(62,534)\\
\includegraphics[scale=1.0]{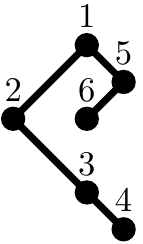}
\end{array}
\begin{array}{c}
(62,354)\\
\includegraphics[scale=1.0]{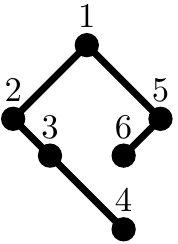}
\end{array}
\begin{array}{c}
(62,345)\\
\includegraphics[scale=1.0]{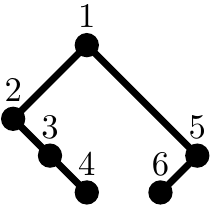}
\end{array}
$$
\vspace{-.5cm}
\caption{
Non-ambiguous trees of the tree $T$ of Figure~\ref{ex:posets}
\label{anac_from_posets}
}
\end{figure}

\begin{proof}
Given a tree $T$, consider the map $\Phi_T$ defined on the set of non-ambiguous trees with underlying tree $T$ by $\Phi_T(A):=(\alpha_L,\alpha_R)$,
where $\Phi(A)=(T,\alpha_L,\alpha_R)$.
We prove
that the image of $\Phi_T$ is $\mathcal{L}(V_L)\times\mathcal{L}(V_R)$,
where we denote by $\mathcal{L}(P)$ the set of linear extensions of a poset $P$.
The second statement in Lemma~\ref{lemma:crux} will then follow, since we deduce from the injectivity of $\Phi$ (Lemma~\ref{Phi}) that $\Phi_T$ is also injective.

First, we prove that $Im\,\Phi_T\subseteq\mathcal{L}(V_L)\times\mathcal{L}(V_R)$.
Without loss of generality, we will prove that $\alpha_L\in\mathcal{L}(V_L)$.
We need to prove that, if $s <_{V_L} t$, then $s$ precedes $t$ in $\alpha_L$, which we shall write $s <_{\alpha_L} t$.
If $s <_{V_L} t$, there exists a path in $T$ starting from $s$ and ending at $t$.
When we go through the path, the $X$-coordinates  of the vertices remain unchanged along  right edges,
while they increase along  left edges.
Since $s\neq t$, we have $X(s)<X(t)$, which is equivalent to $s <_{\alpha_L} t$.

Now the hard part is to prove that $\mathcal{L}(V_L)\times\mathcal{L}(V_R)\subseteq Im\,\Phi_T$.
Let $(\alpha_L,\alpha_R)\in \mathcal{L}(V_L)\times\mathcal{L}(V_R)$.
From the triple $(T,\alpha_L,\alpha_R)$, we may build a set of points in the grid, denoted $A$, as follows:
we place a point (the root) at position $(0,0)$ and one point for every vertex $v$ in $T$ at coordinates
$$
\left\{
\begin{array}{l}
    X(v) = i  \hspace{.2 cm} \text{with}  \hspace{.2 cm} \alpha_L(i)=v \hspace{.4 cm} \text{and} \hspace{.4 cm} Y(v) = Y(\text{parent}(v))  \hspace{.6 cm} \text{ if }v\in V_L; \\
    X(v)=X(\text{parent}(v))  \hspace{.4 cm} \text{and} \hspace{.4 cm}Y(v)=j  \hspace{.2 cm} \text{with}  \hspace{.2 cm} \alpha_R(j)=v    \hspace{.6 cm} \text{ if }v\in V_R.
\end{array}
\right.
$$

We now prove that $A$ is a non-ambiguous tree, whose underlying binary tree is of course $T$. This will follow from the three following statements, that we prove below.
\begin{itemize}
\item[$i.$] for every left (resp. right) edge $(s,t)$ of $T$, we have $X(s) < X(t)$ (resp. $Y(s) < Y(t)$) in $A$;
\item[$ii.$] $A$ avoids the pattern $\pattern$;
\item[$iii.$] two different vertices in $T$ correspond to points at different positions in $A$.
\end{itemize}
\emph{Proof of $i.$} Without loss of generality, we show this property for the set of left edges. First of all, we define the \emph{oldest left (resp. right) ancestor} of a vertex $u$ to be the vertex $v$ such that the path going from $v$ to $u$ contains only right (resp. left) edges and this path is the longest with this property.\label{def:oldest_right_ancestor}

For example, in Figure~\ref{oldest_ancestor} the vertices named $v$ are the oldest left ancestors of those named $u$.

\begin{figure}[H]
$$
\begin{array}{c}
	\includegraphics[scale=0.8]{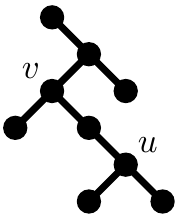}
\end{array}
\hspace{1cm}
\begin{array}{c}
	\includegraphics[scale=0.8]{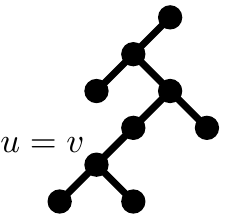}
\end{array}
$$
\caption{
The oldest left ancestors
\label{oldest_ancestor}
}
\end{figure}

Let $v$ be the oldest left ancestor of $s$ in $T$. Our construction of the set $A$ implies that $X(v)=X(s)$. Moreover, since $(s,t)$ is a left edge, we have
$\alpha_L(X(t))=t$. Keeping in mind these two facts, we consider the following two cases:
\begin{itemize}
\item if $v$ is the root of the tree, then $X(v)=0$. Since $X(t)\in[1,n]$, we have $0=X(s)<X(t)$;
\item otherwise, $v$ is the ending point of a left edge in $T$, and $\alpha_L(X(v))=v$. Since in $T$ there exists a path form $v$ to $t$, then $v<_{V_L} t$. Moreover, we know that $\alpha_L\in\mathcal{L}(V_L)$, so $v <_{\alpha_L} t$ and $X(v)<X(t)$, implying that $X(s)<X(t)$.
\end{itemize}

\noindent \emph{Proof of $ii.$} We proceed by way of contradiction. Suppose that there are three points $s,t,u\in A$ such that $s-t-u$ form the pattern $\begin{array}{c}\includegraphics[scale=.8]{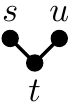}\end{array}$.
We remark that these three points are different from the root, because for each of them either the $X$ or the $Y$-coordinate is nonzero. Without loss of generality, we suppose that $t$ is the end of a left edge. Let $v$ be the oldest left ancestor of $s$. By construction, $Y(v)\leq Y(s)<Y(t)$, and hence $t\neq v$.
Since $X(v)=X(s)\neq 0$, 
$v$ is the end of a left edge.
The points $v$ and $t$ therefore both belong to $V_L$ and have the same $X$-coordinate, enventhough they are distinct.
This is impossible, since $\alpha_L$ is a linear extension of $V_L$, and hence all points in $V_L$ must have different $X$-coordinates.\\

\noindent \emph{Proof of $iii.$} We proceed by way of contradiction. Suppose that there are two different vertices $s,t\in T$ whose positions in $A$ coincide. By construction, they cannot be the root of the tree. We have two cases:
\begin{itemize}
\item one of the vertices is the end of a left edge, and the other is the end of a right edge. In this case, we find an occurrence of the forbidden pattern $\pattern$;
\item $s$ and $t$ are both the end of a left edge (the other case is similar). In this case, we have $s,t\in V_L$ and $X(s)=X(t)$ which is impossible.
\end{itemize}

\vspace{-21pt}
\end{proof}

Now we come to the final step toward proving Proposition~\ref{prop:hook}.
\begin{lemma}\label{lem:Hasse}
The Hasse diagrams of $V_L$ and $V_R$ are forests.
\end{lemma}

\begin{proof}
We prove this proposition by way of contradiction.
Suppose that there is a cycle in the Hasse diagram of $V_R$ (the case of $V_L$ is analogous).
We can deduce from the poset structure that there are two paths in $V_R$ starting from an element $v$ and ending at $w$.
This would imply that in the tree there are two different paths from $v$ to $w$, and hence there would be a cycle in the tree.
\end{proof}

Figure~\ref{hasse_diagram_v_r_v_l} shows an example of the forests obtained by computing the Hasse diagrams of $V_L$ and $V_R$.

\begin{figure}[H]
$$
T =
\begin{array}{c}
\includegraphics[scale=0.8]{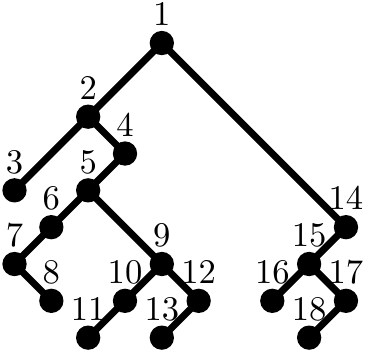}
\end{array}
\hspace{.4cm}
H(V_L) =
\begin{array}{c}
\includegraphics[scale=0.8]{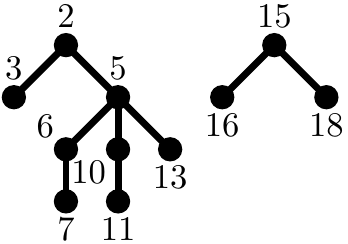}
\end{array}
\hspace{.4cm}
H(V_R) =
\begin{array}{c}
\includegraphics[scale=0.8]{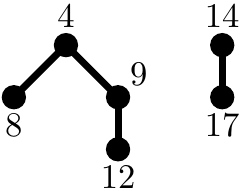}
\end{array}
$$
\vspace{-.5cm}
\caption{
The Hasse diagrams $H(V_L)$ and $H(V_R)$ of $V_L$ and $V_R$ are forests
\label{hasse_diagram_v_r_v_l}
}
\end{figure}

\begin{proposition}\label{prop:hook}
The number of non-ambiguous trees with underlying tree $T$ is given by
\begin{equation}\label{eqn:hook_formula}
NA(T)=\frac{\#\{\textrm{left edges}\}!\,\,\#\{\textrm{right edges}\}!}{\displaystyle{\prod_{e\in V_L} n_e}\,\,\displaystyle{\prod_{e\in V_R} n_e}}\
\end{equation}
where, for every left edge (resp. right edge) $e$, $n_e$ is the number of left edges (resp. right edges)
contained in the subtree whose root is the ending point of $e$, plus $1$.
\end{proposition}

\begin{proof}
Recall that Knuth's hook formula~\cite{knuth} gives the number of linear extensions of a poset $V$ whose Hasse diagram is a forest:
namely, the number of these linear extensions is $|V|!$ divided by the product of the hook lengths of all vertices in $V$.
Therefore, by Lemmas~\ref{lemma:crux} and~\ref{lem:Hasse}, the number of non-ambiguous trees with underlying tree $T$
is given by the product of the results of Knuth's hook formula
applied to the Hasse diagrams of $V_L$ and $V_R$.
Specifically, when applying Knuth's formula to $V_L$ (resp. $V_R$), the hook length of any vertex $v$ in $V_L$ (resp. $V_R$) is the number of descendants of $v$ in $V_L$ (resp. $V_R$) including $v$, which corresponds exactly to $n_e$ for the left (resp. right) edge $e$ whose end point is $v$.
\end{proof}

This new hook formula is illustrated by Figure~\ref{fig:hook}.
\begin{figure}[H]
$$
\hookformula{
\begin{array}{c}
\includegraphics[scale=1.0]{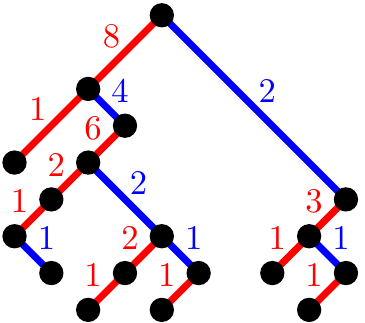}
\end{array}
} = \frac{ \textcolor{\leftedgecolor}{ 11! } \,\, \cdot \,\, \textcolor{\rightedgecolor}{ 6! } }{ \textcolor{\leftedgecolor}{ 1\cdot 1\cdot 1\cdot 1\cdot 1\cdot 1\cdot 2\cdot 2\cdot 3\cdot 6\cdot 8} \,\, \cdot \,\, \textcolor{\rightedgecolor}{ 1\cdot 1\cdot 1\cdot 2\cdot 2\cdot 4 }  }
$$
\vspace{-.5cm}
\caption{
A hook formula for non-ambiguous trees on an example
\label{fig:hook}
}
\end{figure}

%
%

\section{Complete non-ambiguous trees and combinatorial identities}
\label{enumeration_anac}

A non-ambiguous tree is \emph{complete} whenever its vertices have either $0$ or $2$ children.
An example of complete non-ambiguous tree can be found in Figure~\ref{root_suppression}.
A complete non-ambiguous tree always has an odd number of vertices. Moreover, as in complete binary trees, a complete non-ambiguous tree
with $2k+1$ vertices has exactly $k$ internal vertices, $k+1$ leaves, $k$ right edges and $k$ left edges.
Denote by $\nbanac{k}$ the number of complete non-ambiguous trees with $k$ internal vertices.
The sequence $(b_k)_{k\ge 0}$ is known in~\cite{oeis} as ${\tt A002190} = [1, 1, 4, 33, 456, 9460, \dots]$.
We give in this section the first combinatorial interpretation for this integer sequence,  in terms of complete non-ambiguous trees.
Moreover, we use this interpretation to give in Propositions~\ref{carl1} and~\ref{carl2}
combinatorial proofs of  two identities due to Carlitz~\cite{car}.

\subsection{Enumeration of complete non-ambiguous trees, and connection to Bessel function}

Denote by $\Catalan{n}$ the number of complete binary trees with $n$ internal vertices. It is well-known that $\Catalan{n}=\frac{1}{n+1}{2n \choose n}$ is the $n$-th Catalan number, and that, for every $n\geq 0$, we have the identity:
\begin{equation}
\Catalan{n+1} = \sum_{i+j=n} \Catalan{i}\Catalan{j}.
\label{cata}
\end{equation}

Proposition~\ref{carl1} gives a variant of this identity for complete non-ambiguous trees:

\begin{proposition}\label{carl1}For every $n\geq 0$, we have:
\begin{equation}
\nbanac{n+1} = \sum_{i+j=n} {n+1\choose i}{n+1 \choose j}\nbanac{i}\, \nbanac{j}.
\label{idII}
\end{equation}
\end{proposition}

\begin{proof}
The proof of this proposition is similar to the classical proof of~\pref{cata}: the left (resp. right) subtree $A_L$ (resp. $A_R$) of a complete non-ambiguous tree $A$ with $n+1$ internal vertices is a complete non-ambiguous tree with $i$ (resp. $j$) internal vertices, where $i+j=n$.

Figure~\ref{root_suppression} shows an example of left and right subtree of a complete non-ambiguous tree.
\begin{figure}[H]
$$
A =
\begin{array}{c}
\includegraphics[scale=0.8]{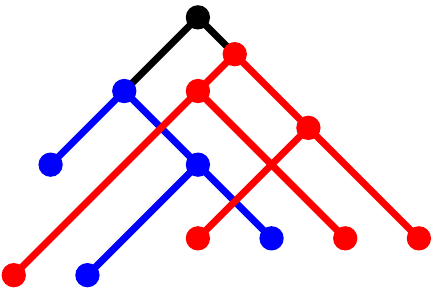}
\end{array}
\hspace{.2cm}
\longrightarrow
\hspace{.2cm}
A_L =
\begin{array}{c}
\includegraphics[scale=0.8]{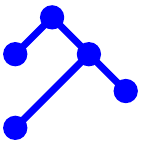}
\end{array}
\hspace{.4cm}
A_R =
\begin{array}{c}
\includegraphics[scale=0.8]{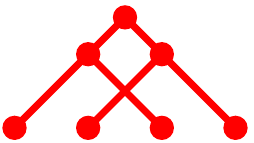}
\end{array}
$$
\vspace{-.5cm}
\caption{
The root suppression in a complete non-ambiguous tree
\label{root_suppression}
}
\end{figure}

Hence, in order to construct an arbitrary complete non-ambiguous tree $A$ with $n+1$ internal vertices, we need to choose:
\begin{itemize}
\item the number $i$ of internal vertices contained in $A_L$ ($i$ may range between $0$ and $n$, the number $j$ is equal to $n-i$);
\item the complete non-ambiguous tree structure of $A_L$ (resp. $A_R$) -- we have $\nbanac{i}$ (resp. $\nbanac{j}$) choices;
\item the way of interlacing the right (resp. left) edges of $A_L$ and $A_R$.
\end{itemize}
We denote by $u_1,u_2,\ldots,u_i$ (resp. $v_1,v_2,\ldots,v_{j}$) the end points of the right edges in $A_L$ (resp. $A_R$) such that if $k<l$, then $Y(u_k)<Y(u_l)$ (resp. $Y(v_k)<Y(v_l)$), and by $u_0$ and $v_0$ the roots of $A_L$ and $A_R$. Now, if we want to interlace the right edges in $A_L$ with those in $A_R$, we need to decide at what positions we want to insert the vertices $u_1,u_2,\ldots,u_i$ with respect to $v_0,v_1,v_2,\ldots,v_{j}$, saving the relative order among $u_0,u_1,u_2,\ldots,u_i$ and $v_0,v_1,v_2,\ldots,v_{j}$. A vertex $u_k$ can be placed either to the left of $v_0$, or between $v_t$ and $v_{t+1}$ ($0\leq t\leq j-1$), or to the right of $v_{j}$.

Hence, we must choose the $i$ positions of $u_1,u_2,\ldots,u_i$ (multiple choices of the same position are allowed) among $j+2$ possible ones. This shows that there are
$\left(\left({j+2\atop i}\right)\right)={i+j+1 \choose i}={n+1 \choose i}$
ways of interlacing the right edges of the subtrees $A_L$ and $A_R$, where $\left(\left({a\atop b}\right)\right)$ denotes the number of way of choosing $b$ objects within $a$, with possible repetitions.

\noindent Analogous arguments apply to left edges. In this case, we have
$\left(\left({i+2\atop j}\right)\right)={n+1 \choose j}$
different interlacements. This ends the proof.
\end{proof}

\begin{corollary}\label{corollarydeux}
The sequence $\nbanac{k}$ satisfies the following identity
\begin{equation}
\sum_{k\geq 0}\nbanac{k}\frac{x^{2(k+1)}}{((k+1)! 2^{k+1})^2} = - \ln \left( \Bessel{0}{x} \right).
\label{bessel}
\end{equation}
\end{corollary}

\begin{proof}
It is well known (see, e.g.,~\cite{abst64}) that the Bessel function $\Bessel{0}{x}=\displaystyle{\sum_{k\geq 0}j_k x^k}$ satisfies the differential equation
\begin{equation}
\frac{d^2 y}{dx^2} + \frac{1}{x} \frac{d \, y}{dx} + y = 0,
\label{equadiff}\end{equation}
The first coefficients in its series expansion are $j_0=1$ and $j_1=0$.

Consider now the function $
B(x) = \displaystyle{\Exponential{ - \sum_{k\geq 0}\nbanac{k}\frac{x^{2(k+1)}}{((k+1)! 2^{k+1})^2}}}=\displaystyle{\sum_{k\geq 0} \beta_k x^k}$.
Equation~\pref{idII} ensures that $B(x)$ satisfies Equation~\pref{equadiff}, \emph{i.e.} the same second order differential equation as $\Bessel{0}{x}$.

Setting $x=0$, we have $\beta_0=B(0)=1=j_0$.
Moreover, in $Z(x)= - \displaystyle{\sum_{k\geq 0}\nbanac{k}\frac{x^{2(k+1)}}{((k+1)! 2^{k+1})^2}}$ only the even powers of $x$ have non-zero coefficients. Hence, since $B(x)=\Exponential{Z(x)}=\displaystyle{\sum_{k\geq 0}\frac{Z(x)^k}{k!}}$, we have $\beta_{2i+1}=0$ for every $i\geq 0$. In particular, $\beta_1=0=j_1$.
These arguments imply that $B(x)=\Bessel{0}{x}$.
\end{proof}

%
%
%

\subsection{Combinatorial identities}

\noindent  Corollary~\ref{corollarydeux} shows that non-ambiguous trees provide
 a combinatorial interpretation -- and to our knowledge, the first one -- of sequence {\tt A002190}~\cite{oeis}.

\noindent In~\cite{car}, the author shows analytically that  identities (\ref{idII}) and~\pref{idIII} below are equivalent.
We give a combinatorial proof of this fact.

\begin{proposition}\label{carl2} For every $n\geq 1$, we have:
\begin{equation}
\sum_{k=0}^{n-1}(-1)^k{n\choose k+1}{n-1\choose k}\nbanac{k}=1.\label{idIII}\end{equation}
\end{proposition}

\begin{proof}
We fix an integer $n$ and we take $0\leq k\leq n-1$.
We define a gridded tree of size $(k,n)$ to be a set of $2k+1$ points placed in a $n\times n$ grid, such that Condition~\ref{condition_2_ana} defining non-ambiguous trees is satisfied (which means we consider a non-ambiguous tree of size $2k+1$ embedded in a $n\times n$ grid) and such that the underlying tree is complete and that its root belongs to the first column.
This implies that there are $n-k-1$ empty columns and $n-k-1$ empty rows, and that the first column is not empty.
Figure~\ref{danslagrille} shows an example of a gridded tree of size $(2,6)$.
\begin{figure}[ht]
\begin{center}
\vspace{-.5cm}
$$
\begin{array}{c}
\includegraphics[scale=0.7]{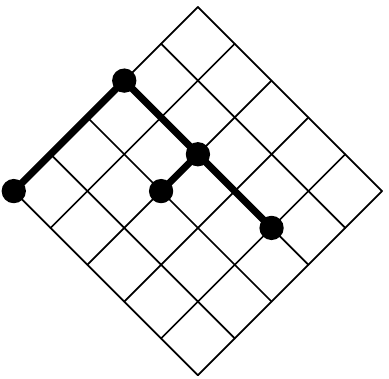}
\end{array}
$$
\end{center}
\vspace{-.5cm}
\caption{
An example of gridded tree with $2$ internal vertices drawn on a $6\times 6$ grid
\label{danslagrille}
}
\end{figure}
%

\noindent It is easy to verify that there are ${n\choose k+1}{n-1\choose k}b_k$ gridded trees of size $(k,n)$. We call trivial gridded tree the tree of size $(0,n)$ consisting of a single vertex in $(0,0)$.
Now, for every integer $n$, we define an involution on the set of non trivial gridded trees. This involution associates a gridded tree of size $(k,n)$ with a gridded tree either of size $(k-1,n)$ or $(k+1,n)$.

\noindent To define this involution, consider a gridded tree of size $(k,n)$ and add a virtual root at position $(-1,0)$; the previous root becomes the left child of the virtual root. Now consider the path starting from the virtual root, going down through the tree, turning at each internal vertex, and ending at a leaf. This path is unique. There are two cases:
\begin{enumerate}
\item \label{involution:erase_leaf} the path does not cross an empty row, nor an empty column: we erase the leaf and its parent from the tree, getting a new gridded tree of size $(k-1,n)$. We can always erase the leaf and its parent, except if the parent is the virtual root. This happens only if the tree is the trivial gridded tree. As we restricted to non trivial gridded trees, this case never happens.
\item \label{involution:add_leaf} the path crosses an empty row or an empty column: we choose the first empty row or column met while visiting the path. Without loss of generality, we suppose that it is a column, say $c$. Then, we add a new vertex $v$ at the position where $c$ crosses the path, and we add in the same column a new leaf (whose parent is $v$) in the topmost empty row. While visiting the path, we did not meet an empty row. Since there are as many empty rows as empty columns, there is always an empty row below $v$. This operation gives rise to a new gridded tree of size $(k+1,n)$.
\end{enumerate}
Remark that adding (resp. removing) a leaf and its parent $p$ in (resp. from) a gridded tree following the previous algorithm does not remove (resp. add) any empty row or column that crosses the path from the virtual root to $p$. For this reason, this operation is an involution. Figure~\ref{involution_example} shows how the involution acts on two examples.
\end{proof}

\begin{figure}[H]
$$
\begin{array}{c}
\includegraphics[scale=0.7]{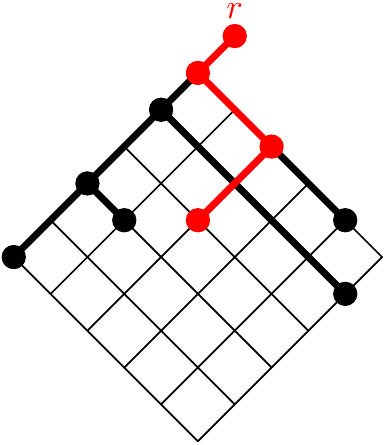}
\end{array}
\begin{array}{c}
\xrightarrow{\hspace{.25cm}\ref{involution:erase_leaf}.\hspace{.25cm}} \\
\xleftarrow[\hspace{.25cm}\ref{involution:add_leaf}.\hspace{.25cm}]{}
\end{array}
\begin{array}{c}
\includegraphics[scale=0.7]{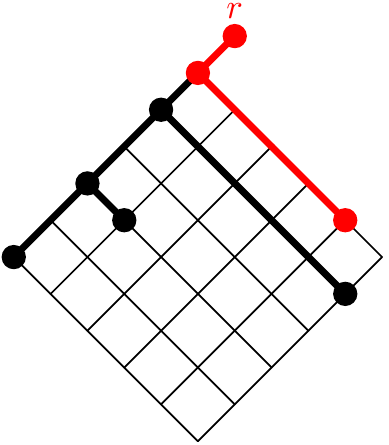}
\end{array} \qquad \begin{array}{c}
\includegraphics[scale=0.7]{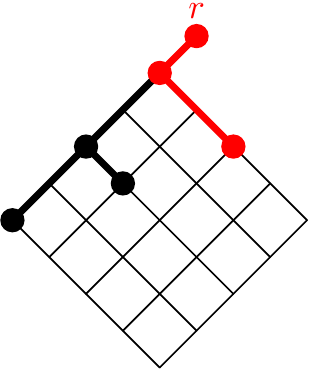}
\end{array}
\begin{array}{c}
\xrightarrow{\hspace{.25cm}\ref{involution:erase_leaf}.\hspace{.25cm}} \\
\xleftarrow[\hspace{.25cm}\ref{involution:add_leaf}.\hspace{.25cm}]{}
\end{array}
\begin{array}{c}
\includegraphics[scale=0.7]{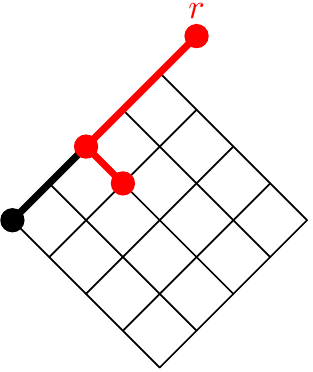}
\end{array}
$$
\vspace{-.5cm}
\caption{
The involution acting on two examples of non trivial gridded trees
\label{involution_example}
}
\end{figure}
%
%

We now state and prove an identity satisfied by Catalan numbers,
whose proof is similar to Proposition \ref{carl2},
and which will be used to get Corollary \ref{identity_ANAC}.

\begin{proposition}\label{identity_Catalan} For every $n\geq 1$, we have:
\begin{equation}
\sum_{k=0}^{n}(-1)^{n+k}{{n+k}\choose {n-k}}  \Catalan{k}=0.\label{id_Catalan}\end{equation}
\end{proposition}

\begin{proof}
We define a {\em $(k,\ell)$-labeled binary tree}  as a complete binary tree with $k$ internal vertices
together with a labeling of its $2k+1$ (internal or external) vertices by non-negative integers,
such that the sum of these labels is equal to $\ell$.
Since ${{n+k}\choose {n-k}}$ is the number of ways to choose $n-k$ elements with repetition
in a set of cardinality $2k+1$,
we may interpret ${{n+k}\choose {n-k}}  \Catalan{k}$ combinatorially as the number of $(k,n-k)$-labeled binary trees.
Indeed, a $(k,n-k)$-labeled binary tree may be described by a complete binary tree with $k$ internal vertices,
together with a choice of $n-k$ vertices (with possible repetition) among its $2k+1$ (internal or external) vertices,
any vertex receiving label $i$ when it has been chosen $i$ times.

Now the proof is based on an involution defined on $(k,n-k)$-labeled binary trees, for $0\le k\le n$,
which modifies the parameter $k$ by $\pm 1$.
This involution is defined  in a way similar to the one used in the proof of Proposition \ref{carl2}.
Let us consider a $(k,n-k)$-labeled binary tree.
We add to it a virtual root  and consider the path starting from the virtual root, going down through the tree,
turning at each internal vertex, and ending at a leaf.
Two cases occur:
\begin{enumerate}
\item \label{involution2:erase_leaf}
If all the vertices on this path are labeled by zero, then we erase the leaf $y$ at the end of this path,
together with its parent, and we increase by one the label of the sibling of $y$.
This produces a $(k-1,n-k+1)$-labeled binary tree.
\item \label{involution2:add_leaf}
Otherwise we consider the first time this path encounters a vertex $v$ that has a non-zero label.
We decrease by one the label of $v$ and we add a new vertex $x$ on the edge between $v$ and its parent,
having two children: $v$ and a new leaf $y$ (the new sibling of $v$).
The labels associated to the new vertices $x$ and $y$ are zero,
so that the tree obtained is a $(k+1,n-k-1)$-labeled binary tree.
\end{enumerate}
Figure~\ref{involution2_example} illustrates this operation on two examples (labels zero are omitted on this figure).
It is clear that we define in this way an involution that modifies the parameter $k$ by $\pm 1$, which implies~\eqref{id_Catalan}.
\end{proof}

\begin{figure}[H]
$$
\begin{array}{c}
\includegraphics[scale=0.7]{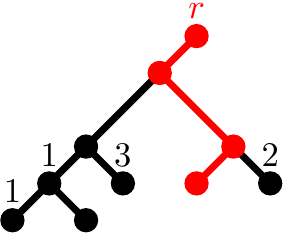}
\end{array}
\begin{array}{c}
\xrightarrow{\hspace{.25cm}\ref{involution:erase_leaf}.\hspace{.25cm}} \\
\xleftarrow[\hspace{.25cm}\ref{involution:add_leaf}.\hspace{.25cm}]{}
\end{array}
\begin{array}{c}
\includegraphics[scale=0.7]{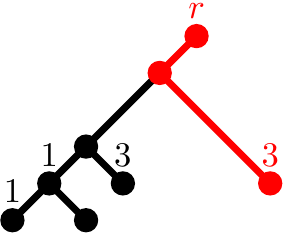}
\end{array} \qquad \begin{array}{c}
\includegraphics[scale=0.7]{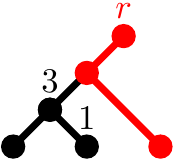}
\end{array}
\begin{array}{c}
\xrightarrow{\hspace{.25cm}\ref{involution:erase_leaf}.\hspace{.25cm}} \\
\xleftarrow[\hspace{.25cm}\ref{involution:add_leaf}.\hspace{.25cm}]{}
\end{array}
\begin{array}{c}
\includegraphics[scale=0.7]{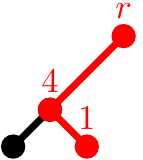}
\end{array}
$$
\vspace{-.5cm}
\caption{
The involution acting on two examples of $(4,7)$-labeled trees
\label{involution2_example}
}
\end{figure}

We are now able to obtain a new identity involving Catalan numbers and the sequence $\nbanac{k}$,
which has been conjectured by P. Hanna (see \cite{oeis}, sequence {\tt A002190}).

\begin{corollary}\label{identity_ANAC} For every $n\geq 1$, we have:
\begin{equation}
\sum_{k=0}^n(-1)^k \nbanac{k}\Catalan{k}{n+k\choose n-k}^2=0.\label{id_ANAC}\end{equation}
\end{corollary}

\begin{proof}
We fix an integer $\myn \geq 1$ and define the matrix $D(\myn)=(D(\myn)_{i,j})_{0\leq i,j\leq \myn}$ as follows: 
$$D(\myn)_{i,j}=\left\{\begin{array}{ll}0&\text{if }i=0\\(-1)^{j}{i+j\choose i-j}^2C_j&\text{otherwise}\end{array}\right. \text{.}$$
We also set ${\bf \underline{b}}(\myn)$ to be the column vector whose entries are $(b_0,b_1,\ldots,b_{\myn})$.
With this notations, Identity \pref{id_ANAC} holds for every $n=1,2 \dots, \myn$  
if and only if $D(\myn)\cdot {\bf \underline{b}}(\myn)={\bf\underline{{0}}}(\myn),$
where ${\bf\underline{{0}}}(\myn)$ is the column vector whose $\myn +1$ entries are all $0$.
Analogously, we define the matrix $A(\myn)=\left((-1)^j { i+1 \choose j+1 }{ i \choose j }\right)_{0\leq i,j\leq \myn}$
and set ${\bf\underline{{1}}}(\myn)$ to be  the column vector whose $\myn +1$ entries are all $1$.
Then, Identity \pref{idIII} for all $n$ from $1$ to $\myn+1$ yields that
$A(\myn)\cdot {\bf \underline{b}}(\myn)={\bf\underline{{1}}}(\myn)$.
Because the matrix $A(\myn)$ is lower triangular with $\pm1$ entries on the diagonal,
it is invertible, and we can write ${\bf \underline{b}}(\myn)=A(\myn)^{-1}\cdot{\bf\underline{{1}}}(\myn)$.
Hence, setting $H(\myn)=\left(H(\myn)_{i,j}\right)_{0\leq i,j\leq \myn}=D(\myn)\cdot A(\myn)^{-1}$,
to conclude the proof of Corollary~\ref{identity_ANAC}, it is enough to prove that:
\begin{equation}H(\myn)\cdot {\bf\underline{{1}}}(\myn)={\bf\underline{{0}}}(\myn).\label{matricial_id}\end{equation}

Notice first that $H(\myn)_{0,j} = 0$ for any $\myn \geq 1$ and $0\leq j\leq \myn$.
This follows immediately from the definition of $H(\myn) = D(\myn)\cdot A(\myn)^{-1}$
and the fact that $D(\myn)_{0,j} =0$ for all $0\leq j\leq \myn$.
The main part of the proof is now to show that for $1\leq i\leq \myn, 0\leq j\leq \myn$, we have
$H(\myn)_{i,j}=(-1)^{i+j}{i+j\choose i-j}C_j$.
Indeed, Proposition~\ref{identity_Catalan} then ensures that
\begin{equation*}\text{for all } i \geq 1, \quad \sum_{j=0}^i (-1)^{i+j}{{i+j}\choose {i-j}}  \Catalan{j}=0, \quad \text{hence, for all } 1 \leq i \leq \myn, \quad \sum_{j=0}^{i} H(\myn)_{i,j} = 0.
\end{equation*}
Notice that for $j>i$, ${i+j\choose i-j}=0$. Consequently, the above equality is equivalent to 
\begin{equation*}
\text{for all } 1 \leq i \leq \myn, \quad \sum_{j=0}^{\myn} H(\myn)_{i,j} = 0.
\end{equation*}
Combined with $H(\myn)_{0,j} = 0$ for all $0\leq j\leq \myn$, this yields $H(\myn)\cdot {\bf\underline{{1}}}(\myn)={\bf\underline{{0}}}(\myn)$ as desired.

\smallskip

So let us focus on proving that for $1\leq i\leq \myn, 0\leq j\leq \myn$,
$H(\myn)_{i,j}=(-1)^{i+j}{i+j\choose i-j}C_j$.

Since $H(\myn)$ is characterized by $H(\myn)\cdot A(\myn)=D(\myn)$,
we have to prove that, for every $i,j$ such that $1\leq i\leq \myn$ and $0\leq j\leq \myn$,
\begin{equation}
\sum_{k=0}^{\myn}(-1)^{i+k}{i+k\choose i-k}{k+1\choose j+1}{k\choose j} C_k={i+j\choose i-j}^2C_j\text{.}\label{product_ex}
\end{equation}
If $k > i$, the summand vanishes. Moreover, it is easy to check that, if $j>i$, both sides in Identity \pref{product_ex} vanish, since all terms are $0$.
Hence, to conclude the proof, it is sufficient to prove that the following identity holds for any $i>0$ and $0 \leq j \leq i$:
\begin{equation}
\sum_{k=0}^{i}(-1)^{i+k}{i+k\choose i-k}{k+1\choose j+1}{k\choose j} C_k={i+j\choose i-j}^2C_j\label{product_id}
\end{equation}
For this purpose, for any $i>0$, $k\geq 0$ and $0\leq j\leq i$, we define
$$F_j(i,k)=\frac{(-1)^{i+k}{i+k\choose i-k}{k+1\choose j+1}{k\choose j} C_k}{{i+j\choose i-j}^2C_j}\text{,}$$
and we show that for any $i>0$ and $0\leq j\leq i$
$$\sum_{k=0}^i F_j(i,k)=1\text{.}$$
In order to prove this identity, we apply the WZ method described in \cite{PWZ}.
We look for a function
$G_j(i,k)$ that satisfies the following identity:
\begin{equation}F_j(i+1,k)-F_j(i,k)=G_j(i,k+1)-G_j(i,k)\text{.}\label{recurrence_wz}\end{equation}
Using the package {\tt Gosper} of {\tt Mathematica}, we determine that $G_j(i,k)$ should be defined by
$$G_j(i,k)=\frac{(-1)^{i+k}{i+k+1\choose i-k+1}{k+1\choose j+1}{k\choose j} C_k}{{i+j\choose i-j}^2C_j}\cdot\frac{2(i+1)(k-j)^2}{(i+j+1)^2(i+k+1)}\text{.}$$
Indeed we then have
$$F_j(i+1,k)-F_j(i,k)-G_j(i,k+1)+G_j(i,k)=w_j(i,k)\cdot\frac{(-1)^{i+k}{i+k+1\choose i-k+1}{k+1\choose j+1}{k\choose j} C_k}{{i+j\choose i-j}^2C_j},$$
where the term $$w_j(i,k)=-\frac{(i-j+1)^2}{(i+j+1)^2}-\frac{i-k+1}{i+k+1}+\frac{2(i+1)(i-k+1)}{(i+j+1)^2}+\frac{2(i+1)(k-j)^2}{(i+k+1)(i+j+1)^2}$$
can be proved to be equal to zero by straightforward computations.

Now, if we sum up both sides of Identity \pref{recurrence_wz} over $0\leq k\leq i+1$,
we can easily see that the right hand side telescopes to $G_j(i,i+2)-G_j(i,0) =0$. Indeed, 
we have that $G_j(i,i+2)=0$, since the binomial coefficient ${i+k+1\choose i-k+1}$ vanishes for $k=i+2$.
And we also have $G_j(i,0)=0$, since ${k\choose j}$ vanishes for $k=0$ and $j>0$, and $(k-j)^2=0$ for $k=j=0$.
This implies that $\displaystyle{\sum_{k=0}^{i+1} \big(F_j(i+1,k)-F_j(i,k)\big) =0}$.
We remark that $F_j(i,i+1) =0$, since ${i+k\choose i-k}$ vanishes if $k=i+1$, and we deduce that $f_j(i):=\displaystyle{\sum_{k=0}^i F_j(i,k)}$ is independent of $i$.
Observing finally that $f_j(j)=F_j(j,j)$ (all other terms vanish, since ${k\choose j} =0$ for $k<j$),
we deduce that $f_j(j)= 1$, and we get the assertion.
\end{proof}

\section{A new bijection between trees and parallelogram polyominoes}
\label{polyominoes}

We recall that a {\pp} of size $n$ is a pair of lattice paths of length $n+1$ with south-west and south-east steps starting at the same point, ending at the same point, and never meeting each other. Figure~\ref{fig:example_pps} shows some examples of \pps of size $4$. The two paths defining a given {\pp} delimit a connected set of boxes. We will consider the {\pp} from this point of view.

\begin{figure}[H]
$$
\begin{array}{c}
\includegraphics[scale=0.8]{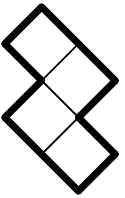}
\end{array}
\hspace{1cm}
\begin{array}{c}
\includegraphics[scale=0.8]{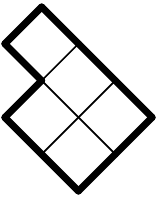}
\end{array}
\hspace{1cm}
\begin{array}{c}
\includegraphics[scale=0.8]{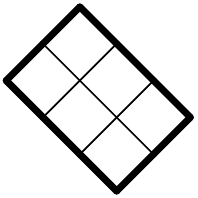}
\end{array}
$$
\vspace{-.5cm}
\caption{Example of \pps of size $4$ \label{fig:example_pps}}
\end{figure}

We now describe a bijection between \pps of size $n$ and binary trees with $n$ vertices by showing that a {\pp} hides a non-ambiguous tree.
In Subsection \ref{sub:bij_pp_trees} we describe an injective map $\Psi$ from \pps to trees.
In Subsection \ref{sub:bij_trees_pp} we describe a map $\Lambda$ from trees to parallelogram polyominoes, and prove in Subsection~\ref{sub:bij_proof} that it is the inverse of $\Psi$.

\subsection{The map $\Psi$ from \pps to trees}
\label{sub:bij_pp_trees}


Given a {\pp} $P$, consider the set $\myS(P)$ of dots defined as follows:
\begin{itemize}
\item we enlighten $P$ from north-west to south-east and from north-east to south-west;
\item we put a dot in the enlightened boxes.
\end{itemize}

\begin{lemma}
For any {\pp} $P$, $\myS(P)$ is a non-ambiguous tree.
\label{lem:is_ana}
\end{lemma}

\begin{proof}
The above construction ensures that it is impossible that all three points in the pattern $\pattern$ are enlightened.
Moreover, only the northernmost box in the {\pp} can be enlightened twice. This implies that every dot (except for the one in the northernmost box) has a parent.
\end{proof}

Let $\Psi$ be the application that associates to any {\pp} of size $n$ the underlying binary tree of $\myS(P)$ (that has $n$ vertices).
An example of this application is shown in Figure~\ref{bijection_polyominoe_arbre}.

\begin{figure}[H]
$$
\begin{array}{c}
\includegraphics[scale=0.8]{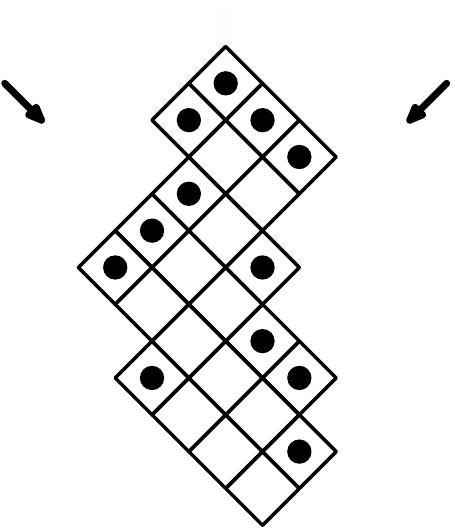}
\end{array}
\xrightarrow[\text{tree}]{\text{non-ambiguous}}
\begin{array}{c}
\includegraphics[scale=0.8]{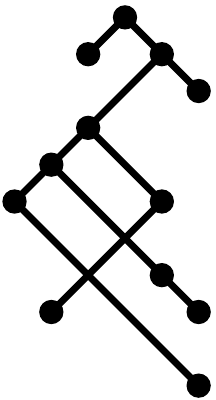}
\end{array}
\xrightarrow[]{\text{tree}}
\begin{array}{c}
\includegraphics[scale=0.8]{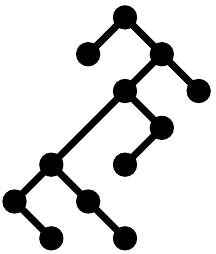}
\end{array}
$$
\vspace{-.5cm}
\caption{
Parallelogram polyominoes are just a way of drawing a binary tree in the plane
\label{bijection_polyominoe_arbre}
}
\end{figure}

\begin{proposition}\label{prop:bij_pp_trees}
The map $\Psi$ is injective.
\end{proposition}

\begin{proof}
Parallelogram polyominoes may be constructed incrementally,
by adding, at every step of the construction process,
one step to each of the two paths defining the {\pp}.
Along the construction of any {\pp} $P$, we may build simultaneously
a non-ambiguous tree whose underlying binary tree is $\Psi(P)$.
More precisely, when adding one step to each of the two paths defining $P$,
we add the enlightened dot(s) corresponding to the inserted steps, when needed.
Figure~\ref{dem_bij_polyo_arbre} shows an example of this construction.

\begin{figure}[H]
\begin{center}
$$
\begin{array}{c}
\includegraphics[scale=1.0]{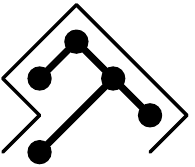}
\end{array}
$$
\end{center}
\caption{An example of {\pp} with its tree under construction\label{dem_bij_polyo_arbre}}
\end{figure}

Considering two different \pps $P$ and $P'$,
there is a first step in the construction process
when the new step added to one path in the construction of $P$
differs from its homologous in $P'$.
We prove that this implies that $\Psi(P) \neq \Psi(P')$.
Indeed, w.l.o.g, the step added to $P$ is SW-oriented, while the corresponding step added to $P'$ is SE-oriented.
This means that only one of these steps is associated with a new dot, connected to its parent $v$.
The dot $v$ exists in both trees $\Psi(P)$ and $\Psi(P')$, but it does not have the same number of children in both trees.
\end{proof}

Since it is known that there are as many \pps of size $n$ as binary trees with $n$ vertices,
we deduce from Proposition~\ref{prop:bij_pp_trees} that the map $\Psi$ is a bijection.
However, it is actually possible to describe the inverse of $\Psi$.
In what follows, we first describe a map $\Lambda$ from binary trees to parallelogram polyominoes,
and we next prove that it is the inverse of $\Psi$.

\subsection{The map $\Lambda$ from trees to \pps}
\label{sub:bij_trees_pp}


The map $\Lambda$ we define here builds a {\pp} of size $n$ from any binary tree $T$ with $n$ vertices.
The main part of the construction is to associate to $T$ a specific non-ambiguous tree $A(T)$
whose underlying binary tree is $T$.
The {\pp} $\Lambda(T)$ is then immediately obtained from $A(T)$.

\subsubsection{Building the non-ambiguous tree $A(T)$}

The main difficulty in describing $A(T)$ is to decide where the points corresponding to the vertices of $T$ should be drawn in the plane.
Like in Subsection~\ref{subsec:hookformula}, their relative position is encapsulated in two total orders
$\alpha_L$ and $\alpha_R$ on the set of vertices of $T$ that are the end of a left or right edge.
We will define these orders in Equation~\eqref{eq:def_order} p.\pageref{eq:def_order},
but it requires some auxiliary definitions and lemmas.

\medskip
%
%

Let $T$ be a tree and $T_r$ be the tree obtained by grafting $T$ as the right child 
of a new vertex~$r$ called the virtual root.
We define two labelings of the vertices of $T$: the turn labeling and the zigzag labeling.

The \emph{turn labeling}, denoted by $t$, is the application on the vertices $s$ of $T$ defined by:
$$
t(s) = \text{the number of turns in the path from $r$ to $s$ in $T_r$}.
$$
For example, in Figure \ref{fig:turning_labeling}, every vertex $s$ is labeled by $t(s)$.

\begin{figure}[H]
\begin{center}
$$
T =
\begin{array}{c}
\includegraphics[scale=0.8]{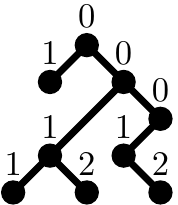}
\end{array}
\hspace{2cm}
T_r =
\begin{array}{c}
\includegraphics[scale=0.8]{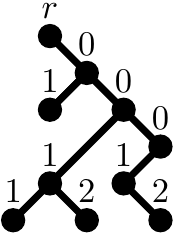}
\end{array}
$$
\end{center}
\caption{The turn labeling \label{fig:turning_labeling}}
\end{figure}

The \emph{zigzag labeling}, denoted by $z$, is the application on the vertices $s$ of $T$ defined by:
$$
z(s) = \left| E(T,s) \right|
$$
where $E(T,s)$ is defined on $T_r$ as follows:
\begin{itemize}
\item let $\mathcal{P}$ be the path from the virtual root $r$ to $s$;
\item let $s_1, s_2, \dots, s_k$ be the vertices of $\mathcal{P}$ where $\mathcal{P}$ is turning;
\item $E(T,s)$ is the set of all vertices $x$ of $T$ except those such that one of the followings holds:
\begin{itemize}
 \item $x \neq s$ is a descendant of $s$, or
 \item $x$ belongs to the right (resp. left) subtree of $s_i$ and the path $\mathcal{P}$ takes a left (resp. right) turn at $s_i$.
\end{itemize}
\end{itemize}
Because $E(T,s)$ can be seen as obtained from $T$ by removing some vertices with all their descendants,
it is straightforward that the vertices remaining in $E(T,s)$ always form an induced subtree of $T$.

To illustrate the above definitions, Figure \ref{fig:zigzag_labeling} shows a tree $T$ with the induced subtree $E(T,s)$ associated to the vertex $s$ of $T$, together with the complete zigzag labeling of $T$.

\begin{figure}[H]
	$$
	T_r =
	\begin{array}{c}
		\includegraphics[scale=0.8]{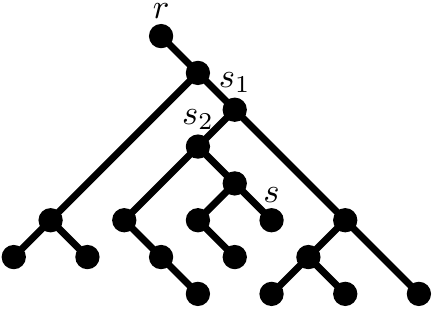}
	\end{array}
	\hspace{.5cm}
	E(T,s) =
	\begin{array}{c}
		\includegraphics[scale=0.8]{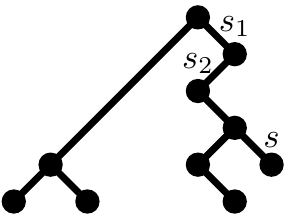}
	\end{array}
	\hspace{.5cm}
	T = \begin{array}{c}
		\includegraphics[scale=0.8]{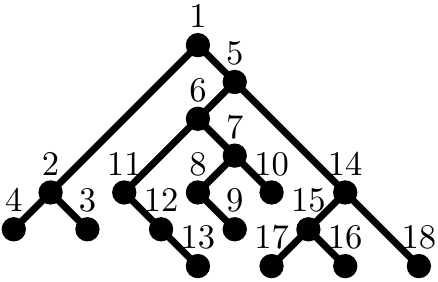}
	\end{array}
	$$
\caption{The zigzag labeling \label{fig:zigzag_labeling}}
\end{figure}

With the turn and zigzag labelings of $T$, we will define (see Equation~\eqref{eq:def_order} p.\pageref{eq:def_order}) the orders $\alpha_L$ and $\alpha_R$ on the vertices of $T$ from which $A(T)$ (and later $\Lambda(T)$) is deduced.

\medskip

An important fact in the construction is that the zigzag labeling is a bijection\footnote{In fact this labeling may also be obtained by performing a kind of "zigzag" depth-first search of $T$.} between the set of vertices of $T$ and $\{1,2, \ldots, n\}$, $n$ denoting the number of vertices of $T$.
This will follow as a consequence of Lemma~\ref{lemma:zigzag} below, for which we introduce some more definitions.


Let $u$ and $v$ be two vertices of a tree $T$.
The \emph{lowest common ancestor} of $u$ and $v$ is the last common vertex between the path going from the root to $u$ and the path going from the root to $v$.
We denote by $LCA(u,v)$ that vertex.

Figure \ref{younger_and_direct_oldest_ancestor} p.\pageref{younger_and_direct_oldest_ancestor} shows an example of lowest common ancestor.

\begin{lemma}
Let $u$ and $v$ be two different vertices of a tree $T$, then
$$E(T,u) \varsubsetneq E(T,v) \text{ or } E(T,v) \varsubsetneq  E(T,u)\text{.}$$
Moreover, $E(T,u) \varsubsetneq  E(T,v)$ holds if and only if the following holds:
$$
\left(
\begin{array}{c}
u = LCA(u,v)
\vspace{.2cm} \\
	\text{ or }
\vspace{.2cm} \\
\left[
	u,v \neq LCA(u,v) \hspace{.50cm} \text{ and } \hspace{.50cm}
	LCA(u,v) \text{ is in } V_L \hspace{.50cm} \text{ and } \hspace{.50cm}
	\begin{array}{c}
		u \text{ is in the right} \\
		\text{subtree of }LCA(u,v)
	\end{array}
\right]
\vspace{.2cm} \\
\text{ or }
\vspace{.2cm} \\
\left[
	u,v \neq LCA(u,v) \hspace{.50cm} \text{ and } \hspace{.50cm}
	LCA(u,v) \text{ is in } V_R \hspace{.50cm} \text{ and } \hspace{.50cm}
	\begin{array}{c}
		u \text{ is in the left} \\
		\text{subtree of }LCA(u,v)
	\end{array}
\right]
\end{array}
\right)
$$
\label{lemma:zigzag}
\end{lemma}
Recall from Subsection~\ref{subsec:hookformula} that $V_L$ (resp. $V_R$) denotes the set of the end points of the left (resp. right) edges of $T$.

\begin{proof}
Let $u,v$ be two different vertices of a tree $T$.
By definition of $LCA(u,v)$, there exists a unique vertex $\alpha \in \{u,v\}$ (we denote by $\beta$ the other vertex such that $\{u,v\} = \{\alpha,\beta\}$) such that one of the three following cases occurs:
\begin{enumerate}
\item \label{case:ancetre} $\alpha = LCA(u,v)$
\item \label{case:cousing} $LCA(u,v) \not\in \{ u,v \}$ and $LCA(u,v) \in V_L$ and $\alpha$ is in the right subtree of $LCA(u,v)$
\item \label{case:cousind} $LCA(u,v) \not\in \{ u,v \}$ and $LCA(u,v) \in V_R$ and $\alpha$ is in the left subtree of $LCA(u,v)$
\end{enumerate}
We will now prove that cases \ref{case:ancetre}, \ref{case:cousing} and \ref{case:cousind} imply $E(T,\alpha) \varsubsetneq E(T,\beta)$.
For the first case, $E(T,\alpha) \varsubsetneq E(T,\beta)$ is a direct consequence of the definition of $E$.
For the two other cases, we just have to remark that
\begin{itemize}
\item the path from the virtual root to $\alpha$ turns at $LCA(\alpha,\beta)$,
\item the path from the virtual root to $\beta$ goes through but doesn't turn at $LCA(\alpha,\beta)$,
\end{itemize}
to conclude that $E(T,\alpha) \varsubsetneq E(T,\beta)$.
Now, Lemma \ref{lemma:zigzag} follows immediately.
\end{proof}

\begin{corollary}
\label{cor:l_bijection}
The labeling $z$ is a bijection between the set of vertices $V$ of $T$ and $\{1,2, \ldots, |V| \}$.
%
Moreover, for all vertices $u$ and $v$ in $V$, we have: $z(u) < z(v) \Longleftrightarrow E(T,u) \varsubsetneq  E(T,v)$.
\end{corollary}

Corollary~\ref{cor:l_bijection} is the key to defining the total orders $\alpha_L$ and $\alpha_R$ on the vertices of $T$, from which $A(T)$ and $\Lambda(T)$ will be derived.

\medskip

We define the partial orders $\alpha_L$ (resp. $\alpha_R$) on the ends of the left (resp. right) edges of $T$ by the relation:
\begin{equation}
\begin{array}{c}
u <_{\alpha_L} v \\
\text{(resp. $u <_{\alpha_R} v$)}
\end{array}
\hspace{1cm}  \text{ if and only if } \hspace{1cm} (t(u),z(u)) <_{\text{lex}} (t(v),z(v)) \label{eq:def_order}
\end{equation}
where $<_{\text{lex}}$ is the lexicographic order relation.
Because $z$ is a bijection, $\alpha_L$ and $\alpha_R$ are actually \emph{total} orders.

\medskip

For any binary tree $T$, we set $A(T)$ to be the triple $(T,\alpha_L,\alpha_R)$.

\begin{lemma}\label{lem:correctness}
For any binary tree $T$, $A(T)$ defines a non-ambiguous tree.
\end{lemma}

\begin{proof}
Let us denote $A(T) = (T,\alpha_L,\alpha_R)$.
We prove that $\alpha_L$ (resp. $\alpha_R$) is a linear extension of the poset $V_L$ (resp. $V_R$) of $T$ (see Subsection~\ref{subsec:hookformula} for the definition of these posets).
By Lemma~\ref{lemma:crux} (p.~\pageref{lemma:crux}), this ensures that $(T, \alpha_L, \alpha_R)$ encodes a non-ambiguous tree of shape $T$.

Let $u$ and $v$ be two vertices of $V_L$ (resp. $V_R$) such that $u <_{V_L} v$ (resp. $u <_{V_R} v$).
By definition, there exists a path from vertex $u$ of $V_L$ (resp. $V_R$) to vertex $v$ of $V_L$ (resp. $V_R$).
We deduce that $t(u) \le t(v)$. Therefore:
\begin{itemize}
\item if $t(u) < t(v)$ then, by definition of $\alpha_L$, $u <_{\alpha_L} v$ (resp. $u <_{\alpha_R} v$ );
\item if $t(u) = t(v)$ then, since there is a path from $u$ to $v$, we have $u = LCA(u,v)$.
We deduce from Lemma \ref{lemma:zigzag} that $z(u) < z(v)$. Hence, $u <_{\alpha_L} v$ (resp. $u<_{\alpha_R} v$).
\end{itemize}
We conclude that $\alpha_L$ (resp. $\alpha_R$) is a linear extension of $V_L$ (resp. $V_R$).
\end{proof}

\subsubsection{From $A(T)$ to the {\pp} $\Lambda(T)$}

In the sequel, by the notation $\begin{array}{c}\includegraphics[scale=0.6]{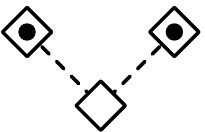}\end{array}$ (resp. $\begin{array}{c}\includegraphics[scale=0.6]{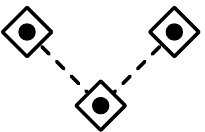}\end{array}$) we mean that there are two cells at coordinates $(x_1,y_1)$ and $(x_2,y_2)$ with $x_1 > x_2$ and $y_1 < y_2$ that contain dots (or vertices), and such that the cell at coordinates $(x_1,y_2)$ does not contain (resp. contains) a dot (or vertex). Notice that it is not necessary that $x_1=x_2+1$ nor that $y_1=y_2-1$ to have an occurrence of either the pattern $\begin{array}{c}\includegraphics[scale=0.6]{images/pattern_3}\end{array}$ or $\begin{array}{c}\includegraphics[scale=0.6]{images/pattern_4}\end{array}$.

For any non-ambiguous tree $A$, consider the operation which
converts every pattern $\begin{array}{c}\includegraphics[scale=0.6]{images/pattern_3}\end{array}$ to a pattern $\begin{array}{c}\includegraphics[scale=0.6]{images/pattern_4}\end{array}$.
We shall call this operation the {\em filling} of $A$, and denote it $F(A)$.

For any binary tree $T$, we define $\Lambda(T) = F(A(T))$.
Figure \ref{fig:bij_tree_pp}, shows all the steps of the construction.
\begin{figure}[H]
	$$
	\begin{array}{c}
		\includegraphics[scale=0.8]{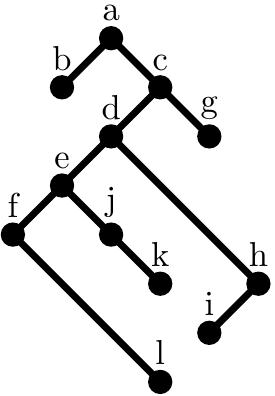}
	\end{array}
	\xrightarrow[(t,z)]{\text{Computation of}}
	\begin{array}{c}
		\includegraphics[scale=0.8]{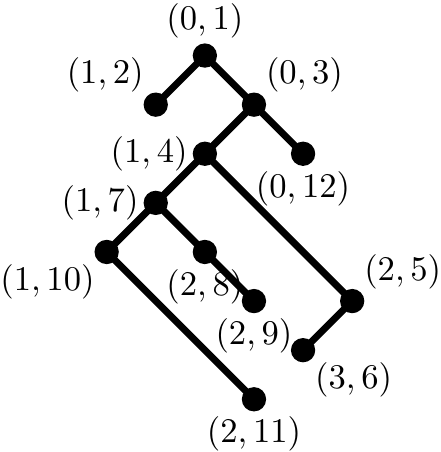}
	\end{array}
	\longrightarrow
	\begin{array}{c}
	\alpha_L = \text{bdefi} \\
	\\
	\alpha_R = \text{cghjkl}
	\end{array}
	\longrightarrow
	$$
	$$
	\begin{array}{c}
		\includegraphics[scale=0.8]{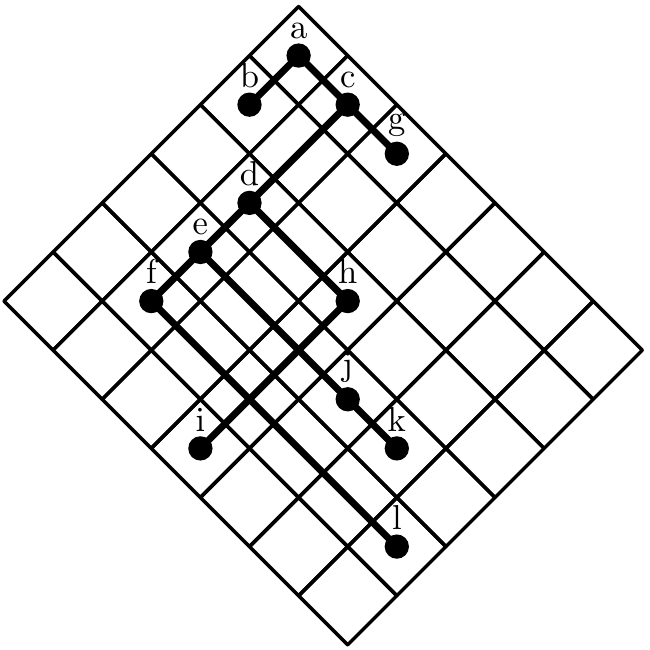}
	\end{array}
	\longrightarrow
	\begin{array}{c}
		\includegraphics[scale=0.8]{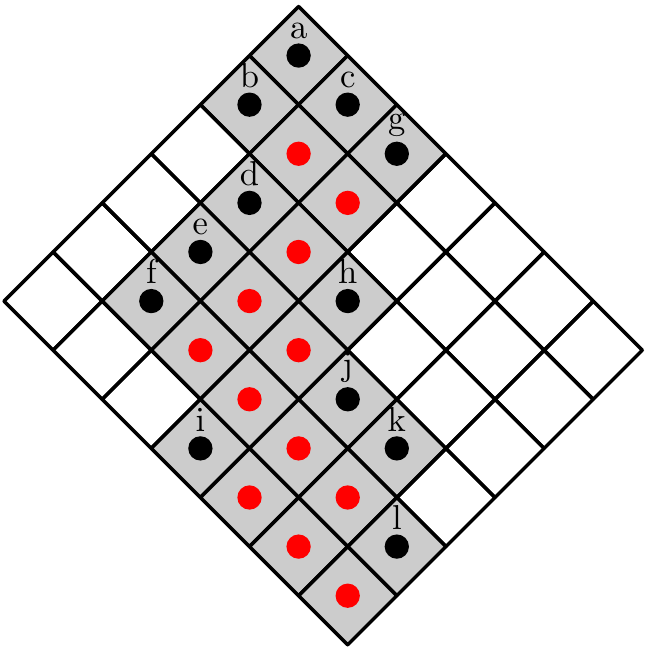}
	\end{array}
	$$
\caption{The map $\Lambda$ \label{fig:bij_tree_pp}}
\end{figure}

Lemma~\ref{lem:ana_that_are_pp} and Proposition~\ref{prop:produces_a_pp} below prove that
for any binary tree $T$, the set of points $\Lambda(T)$ is indeed a parallelogram polyomino.

\begin{lemma}\label{lem:ana_that_are_pp}
Let $A$ be a non-ambiguous tree, and define $V_L$ and $V_R$ as usual.
$F(A)$ is a parallelogram polyomino
as soon as for any two different vertices $u, v$ in $V_L$ (resp. $V_R$) we have
$
X(u) < X(v) \Rightarrow Y(u) \le Y(v)
$
(resp. $Y(u) < Y(v) \Rightarrow X(u) \le X(v)$).
\end{lemma}

\begin{proof}
In the proof of this lemma, 
we use a characterization of parallelogram polyominoes (see \cite{dedudu} for more details):
a parallelogram polyomino is a polyomino $P$ (\emph{i.e.}, a finite set of unit cells on the grid $\mathbb{N}\times\mathbb{N}$ -- oriented as shown on Figure~\ref{grid} -- that is edge-connected)
such that:
\begin{itemize}
\item $P$ is row and column convex, \emph{i.e.} each of its rows and columns is connected;
\item if we consider the minimum rectangle $R$ bounding $P$, then $P$ contains the top-most and bottom-most
corners of $R$.
\end{itemize}

Suppose that for every $u,v\in V_L$ we have $X(u) < X(v) \Rightarrow Y(u) \le Y(v)$ and that, for every $t,w\in V_R$, we have $Y(t) < Y(w) \Rightarrow X(t) \le X(w)$.
Recall that $\{X(v), v\in V_L\}$ is the interval $\{1,\dots,\vert V_L\vert\}$.
Without loss of generality, we will show that $F(A)$ is column convex. Suppose that $F(A)$ contains the two cells $(a,y_0)$ and $(b,y_0)$, $a<b$. We want to show that $F(A)$  contains all the cells $(c,y_0)$, with $a<c<b$. The two cells $(a,y_0)$ and $(b,y_0)$ can correspond either to a vertex of $A$ or to the $0$ of a pattern $\begin{array}{c}\includegraphics[scale=0.6]{images/pattern_3}\end{array}$ in $A$. We study the case where both cells correspond to occurrences of the pattern $\begin{array}{c}\includegraphics[scale=0.6]{images/pattern_3}\end{array}$, the other three possible cases being analogous. Denote by $g$ the southwestern-most vertex of $A$ such that $X(g)<a$ and $Y(g)=y_0$ and by $h$ the southeastern-most vertex of $A$ such that $X(h)=b$ and $Y(h)<y_0$. Let $h'\in V_L$ be the oldest left ancestor of $h$. We choose now a cell $f=(c,y_0)$, where $a<c<b$.
Because $X(h')=b$, $c<b$ and $\{X(v), v\in V_L\}=\{1,\dots,\vert V_L\vert\}$, there exists some vertex
$h''\in V_L$ such that $X(h'')=c$. But by our assumption, we deduce from $X(h'')<X(h')$ that $Y(h'')\leq Y(h')<y_0$. Hence, $h''$, $f$, and $g$ give rise to an occurrence of the pattern $\begin{array}{c}\includegraphics[scale=0.6]{images/pattern_3}\end{array}$, as shown in Figure \ref{fig:anappbis}.

\begin{figure}[H]
\begin{center}
$$
\begin{array}{c}
\includegraphics[scale=1.0]{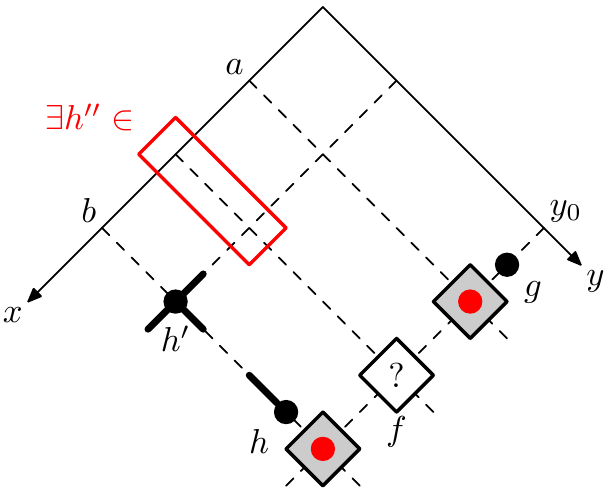}
\end{array}
$$
\end{center}
\caption{The positions of $h''$, $f$ and $g$\label{fig:anappbis}}
\end{figure}

This implies that $f$ belongs to $F(A)$, and hence $F(A)$ is column convex.

Now we will prove that $F(A)$ is connected. Consider two cells $c_1,c_2\in F(A)$. We suppose that both $c_1$ and $c_2$ correspond to the $0$ of two patterns of type $\begin{array}{c}\includegraphics[scale=0.6]{images/pattern_3}\end{array}$, the other possible cases being easier to prove. Denote by $d_i$, $i=1,2$, the southwestern-most vertex of $A$ such that $X(d_i)<X(c_i)$ and $Y(d_i)=Y(c_i)$. Since $F(A)$ is column convex, it contains the vertical strip $S_1$ (resp. $S_2$) joining $d_1$ and $c_1$ (resp. $d_2$ and $c_2$). Now $d_1$ and $d_2$ belong to $A$, and they are connected by a sequence of edges, that correspond to a subset $S$ of $F(A)$. Indeed, $F(A)$ is both row and column convex, and it contains all the cells corresponding to the vertices of $A$. Then, the two cells $c_1$ and $c_2$ are connected by $S_1\cup S_2\cup S$.

Finally, we consider the minimum rectangle $R$ bounding $F(A)$. The root of $A$ occupies the top-most 
corner, which is therefore contained in $F(A)$. Moreover, it is obvious that both the southwestern-most row and the southeastern-most column of $R$ contain a vertex of $A$, and for this reason, the bottom-most 
corner of $R$ is the $0$ in a pattern $\begin{array}{c}\includegraphics[scale=0.6]{images/pattern_3}\end{array}$ of $A$. Therefore, $F(A)$ contains also the bottom-most 
corner of $R$.

All these arguments prove that $F(A)$ is a \pp.
\end{proof}

\begin{remark}
We may prove that the condition in Lemma~\ref{lem:ana_that_are_pp} is also necessary:
if $F(A)$ is a parallelogram polyomino,
then for any two different vertices $u, v$ in $V_L$ (resp. $V_R$) we have
$
X(u) < X(v) \Rightarrow Y(u) \le Y(v)
$
(resp. $Y(u) < Y(v) \Rightarrow X(u) \le X(v)$).
\end{remark}

\begin{proposition}\label{prop:produces_a_pp}
For any binary tree $T$, $\Lambda(T)$ is a parallelogram polyomino.
\end{proposition}


One additional definition is used in our proof of Proposition~\ref{prop:produces_a_pp}.
Let $u$ be the end of a left (resp. right) edge of $T$.
The \emph{direct oldest ancestor} of a vertex $u$ is 
oldest ancestor of $u$ that can reach $u$ using left (resp. right) edges only.
We denote by $DOA(u)$ that vertex.

Figure \ref{younger_and_direct_oldest_ancestor} shows an example of direct oldest ancestor.

\begin{figure}[H]
$$
\begin{array}{c}
	\includegraphics[scale=0.8]{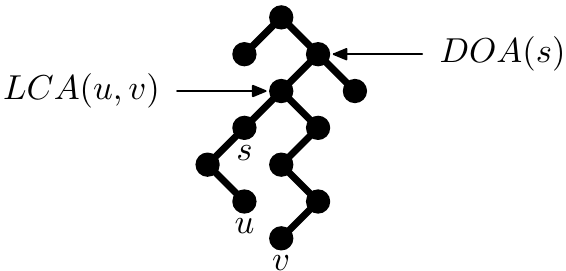}
\end{array}
$$
\caption{
Example of lowest common ancestor and direct oldest ancestor
\label{younger_and_direct_oldest_ancestor}
}
\end{figure}

\begin{proof}
Let $T$ be any binary tree, and set $A(T) = (T, \alpha_L, \alpha_R)$.
Because $\Lambda(T) = F(A(T))$, it is enough to prove that $A(T)$ satisfies the condition of Lemma~\ref{lem:ana_that_are_pp}.
Let $u, v$ be two different vertices of $V_L$ such that $X(u) < X(v)$.
(The case where $u,v \in V_R$ and $Y(u)<Y(v)$ is analogous.)

We want to prove that $Y(u) \le Y(v)$. As $u$ and $v$ are in $V_L$, we have $Y(u) = Y(DOA(u))$ and $Y(v)=Y(DOA(v))$, and we shall rather prove that $Y(DOA(u)) \le Y(DOA(v))$.

If $DOA(u) = DOA(v)$, the above claim is clear. Let us therefore assume that $DOA(u) \neq DOA(v)$.

Because $u$ and $v$ are in $V_L$, we obtain immediately that $t(DOA(u)) = t(u)-1$ and $t(DOA(v))=t(v)-1$. Moreover, since $X(u) < X(v)$ Lemma~\ref{lem:correctness} and the definitions of Subsection~\ref{subsec:hookformula} ensure that $u <_{\alpha_L} v$. We now distinguish two cases, following the definition of $\alpha_L$.

\paragraph*{Case 1: $t(u) < t(v)$.}
\begin{itemize}
 \item If $DOA(u)$ is the root then $Y(DOA(u))=0$ and $Y(DOA(u))\le Y(DOA(v))$.
 \item If $DOA(v)$ is the root, since $t(u)<t(v)$, $DOA(u)$ is the root and we conclude as above.
 \item Otherwise, $DOA(u)$ and $DOA(v)$ are the end of a right edge. From $t(u) < t(v)$, we obtain that $t(DOA(u)) < t(DOA(v))$. Hence $DOA(u) <_{\alpha_R} DOA(v)$. As before, we deduce that $Y(DOA(u)) < Y(DOA(v))$.
\end{itemize}

\paragraph*{Case 2: $t(u)=t(v)$ and $z(u)<z(v)$.} ~\\
Since $t(u)=t(v)$ we have $t(DOA(u)) = t( DOA(v) )$.
We deduce that the following equivalence holds: $z(DOA(u)) < z(DOA(v)) \Leftrightarrow DOA(u) <_{\alpha_R} DOA(v) \Leftrightarrow Y(DOA(u)) < Y(DOA(v))$.
We shall therefore prove that $z(DOA(u)) < z(DOA(v))$.
This will follow from the two claims below:
\begin{enumerate}
  \item \label{cond:a} $LCA(u,v) \neq u$ and $LCA(u,v) \neq v$.
  \item \label{cond:b} $LCA( u,v ) = LCA( DOA(u), DOA(v) )$.
\end{enumerate}

\emph{Proof of Claim}~\ref{cond:a}. As $u,v$ are in $V_L$, $u$ and $v$ can't be the root. We deduce that $DOA(u) \neq u$ and $DOA(v) \neq v$. Because we assumed that $DOA(u) \neq DOA(v)$ and $t(u)=t(v)$, we deduce that $LCA(u,v) \neq u$ and $LCA(u,v) \neq v$.

\emph{Proof of Claim}~\ref{cond:b}. This is also a consequence of $DOA(u) \neq DOA(v)$ and $t(u)=t(v)$.

\begin{itemize}
 \item If $LCA( DOA(u),DOA(v) ) \neq DOA(u)$ and $LCA( DOA(u), DOA(v) ) \neq DOA(v)$,
we apply Lemma \ref{lemma:zigzag} simultaneously on $u$ and $v$ on one side, and on $DOA(u)$ and $DOA(v)$ on the other side. Because $LCA( u,v ) = LCA( DOA(u), DOA(v) )$, we deduce that:
$E(T,u) \varsubsetneq  E(T,v) \Leftrightarrow E(T,DOA(u)) \varsubsetneq  E(T,DOA(v))$. With Corollary~\ref{cor:l_bijection}, this is equivalent to
$z(u) < z(v) \Leftrightarrow  z(DOA(u)) < z(DOA(v))$. As $z(u)<z(v)$, we conclude that $z(DOA(u)) < z(DOA(v))$.
 \item If $LCA( DOA(u),DOA(v) ) = DOA(u)$, then
we apply Lemma \ref{lemma:zigzag} on $DOA(u)$ and $DOA(v)$ to deduce that $E(T,DOA(u)) \varsubsetneq  E(T,DOA(v))$, hence $z(DOA(u))<z(DOA(v))$.
 \item We finally prove that the case $LCA( DOA(u),DOA(v) ) = DOA(v)$ is not possible. Indeed, assume that $LCA( DOA(u),DOA(v) ) = DOA(v)$. If $DOA(v)$ were the root of $T$, then we would have $z(DOA(v)) =1$, which is a contradiction to $z(u)<z(v)$. So $DOA(v)$ is the end of a left (resp. right) edge. Hence by definition, $v$ is in the right (resp. left) subtree of $DOA(v)$. But from Claim~\ref{cond:b}, we have $LCA( u,v ) = DOA(v)$. Lemma~\ref{lemma:zigzag} then ensures that $z(v)<z(u)$, which is a contradiction to $z(u)<z(v)$. \qedhere
\end{itemize}
\end{proof}

\subsection{Proving that $\Lambda$ is the inverse of $\Psi$}\label{sub:bij_proof}

We want to prove that the maps $\Psi$ and $\Lambda$ defined in the previous subsections are inverse of one another. First of all, we define a \emph{ diagram} to be a finite collection of cells in the square grid $\mathbb{N}\times\mathbb{N}$. Then, define the set $D_0$ (resp. $D_1$) of all the  diagrams that avoid the pattern $\begin{array}{c}\includegraphics[scale=0.6]{images/pattern_3}\end{array}$ (resp. $\begin{array}{c}\includegraphics[scale=0.6]{images/pattern_4}\end{array}$). We remark that every \pp\ belongs to $D_0$, since it is convex and contains the bottom-most point of the minimum rectangle bounding it. Moreover, every non-ambiguous tree belongs to $D_1$ by definition.

We extend the map $\myS$ defined in Subsection \ref{sub:bij_pp_trees} to the set $D_0$, defining $\hat{\myS}$ as follows:
for every diagram $U\in D_0$, $\hat{\myS}(U)$ is the diagram obtained
by replacing every $1$ of $U$ occurring as the bottom-most $1$ of a pattern $\begin{array}{c}\includegraphics[scale=0.6]{images/pattern_4}\end{array}$ by a $0$.
Analogously, we can define the extension $\hat{F}$ of $F$ to the set $D_1$:
for every diagram $W\in D_1$, $\hat{F}(W)$ is the diagram obtained by replacing each occurrence of the pattern $\begin{array}{c}\includegraphics[scale=0.6]{images/pattern_3}\end{array}$ in $W$ with an occurrence of $\begin{array}{c}\includegraphics[scale=0.6]{images/pattern_4}\end{array}$.
Obviously, we have that $\hat{\myS}:D_0\to D_1$, $\hat{F}:D_1\to D_0$, $\hat{\myS}|_{\mathcal{P}}=S$, and $\hat{F}|_{Im(A)}=F$, where $\mathcal{P}$ is the set of all \pps and $Im(A)$ is the image of the map $T\mapsto A(T)$. By definition of $\hat{\myS}$ and $\hat{F}$, we easily deduce that:

\begin{lemma}\label{lem:extensions}
$\hat{\myS}=\hat{F}^{-1}$.
\end{lemma}
%
%
%

The final step toward proving that $\Psi$ and $\Lambda$ are inverse of one another is the following lemma:

\begin{lemma}\label{lem:inverse}
For any binary tree $T$, we have $A(T) = \myS(\Lambda(T))$.
\end{lemma}

\begin{proof}
Let $T$ be a binary tree. Proposition \ref{prop:produces_a_pp} states that $\Lambda(T) = F(A(T))$ is a \pp, hence $\myS(\Lambda(T))$ is well defined.
Then, Lemma \ref{lem:extensions} implies that
\par
\vspace{\abovedisplayskip}
\hfill $\displaystyle \myS(\Lambda(T))=\myS(F(A(T)))=\hat{\myS}(\hat{F}(A(T)))=A(T)\text{.}$ \qedhere
\end{proof}

\begin{proposition}
$\Psi$ and $\Lambda$ are inverse of one another.
\end{proposition}

\begin{proof}
By Lemma~\ref{lem:inverse}, we get that $\Psi \circ \Lambda$ is the identity over the set of binary trees. This implies in particular that $\Psi$ is surjective, and that $\Lambda$ is injective.
Recall from Proposition~\ref{prop:bij_pp_trees} that $\Psi$ is also injective, so that we obtain that $\Psi$ is a bijection whose inverse is $\Lambda$.
\end{proof}

\section{Perspectives}
\label{sec:open}
In this work, we used the notion of non-ambiguous tree for the study of various combinatorial problems.
Several questions related to these objects remain open.

The first perspective concerns Bessel functions. We recall that parallelogram polyominoes
are enumerated with respect to their area, width and height through the quotient
of analogues of the Bessel functions $J_0$ and $J_1$ \cite{mbmV}.
This {\em double}  appearance of Bessel functions in close contexts is striking,
but unexplained for the moment.
In the same vein, recall that we were able to build combinatorial objects which
interpret combinatorially the integers $b_k$ that appear in the development of $J_0$ --
namely complete non-ambiguous trees. Is it possible to generalize this construction
to coefficients of general Bessel functions $J_k$?

For what concerns generalizations, an intriguing question would be to
define and study nice analogues of non-ambiguous trees in higher dimensions.
This is indeed work in progress.

We conclude with an open problem of a more algebraic nature.
The Loday-Ronco Hopf algebra $YSym$ of planar binary trees \cite{LR} has basis elements indexed
by planar binary trees. In this setting, two permutations are in the same class when they have the same
underlying tree, when coded as binary search trees.
It would be interesting to study an analogous problem in which we refine these classes
with respect to the underlying {\em non-ambiguous} tree.

\subsection*{Acknowledgements}
\label{sec:ack}
We thank Philippe Nadeau for helpful discussions around rectangular alternative tableaux.
We are also grateful to Christian Krattenthaler and Cyril Banderier for their advice that guided our proof of Corollary~\ref{identity_ANAC}.

This research was driven by computer exploration using the open-source mathematical software \texttt{Sage}~\cite{sage} and its algebraic combinatorics features developed by the \texttt{Sage-Combinat} community~\cite{Sage-Combinat}.

\bibliographystyle{alpha}
\bibliography{bibliography2}

\newcommand{\etalchar}[1]{$^{#1}$}
\begin{thebibliography}{BMV92}

\bibitem[ABN11]{abn}
J.-C. Aval, A.~Boussicault, and P.~Nadeau.
\newblock {Tree-like tableaux}.
\newblock In {\em {DMTCS Proceedings 23rd International Conference on Formal
  Power Series and Algebraic Combinatorics (FPSAC 2011)}}, pages 63--74,
  Islande, 2011. DMTCS.

\bibitem[ABN13]{avbona13}
J.-C. Aval, A.~Boussicault, and P.~Nadeau.
\newblock {Tree-like tableaux. Full version, in preparation}.
\newblock 2013.

\bibitem[AS64]{abst64}
M.~Abramowitz and I.~A. Stegun.
\newblock {\em Handbook of mathematical functions with formulas, graphs, and
  mathematical tables}, volume~55 of {\em National Bureau of Standards Applied
  Mathematics Series}.
\newblock U.S. Government Printing Office, Washington, D.C., 1964.

\bibitem[BMV92]{mbmV}
Mireille Bousquet-M{\'e}lou and Xavier~G{\'e}rard Viennot.
\newblock Empilements de segments et q-{\'e}num{\'e}ration de polyominos
  convexes dirig{\'e}s.
\newblock {\em J. Comb. Theory, Ser. A}, 60(2):196--224, 1992.

\bibitem[Bur07]{Bur07}
A.~Burstein.
\newblock On some properties of permutation tableaux.
\newblock {\em Ann. Comb.}, 11(3-4):355--368, 2007.

\bibitem[Car63]{car}
L.~Carlitz.
\newblock A sequence of integers related to the {B}essel functions.
\newblock {\em Proc. Amer. Math. Soc.}, 14:1--9, 1963.

\bibitem[DDD95]{dedudu}
M.~Delest, J.~P. Dubernard, and I.~Dutour.
\newblock Parallelogram polyominoes and corners.
\newblock {\em J. Symbolic Comput.}, 20(5-6):503--515, 1995.
\newblock Symbolic computation in combinatorics $\Delta{_{1}}$ (Ithaca, NY,
  1993).

\bibitem[DV84]{delvie}
M.-P. Delest and G.~Viennot.
\newblock Algebraic languages and polyominoes enumeration.
\newblock {\em Theoretical Computer Science}, 34(1-2):169 -- 206, 1984.

\bibitem[ES00]{Ehrenborg2000284}
R.~Ehrenborg and E.~Steingr\'imsson.
\newblock The excedance set of a permutation.
\newblock {\em Advances in Applied Mathematics}, 24(3):284 -- 299, 2000.

\bibitem[Knu98]{knuth}
D.~E. Knuth.
\newblock {\em The art of computer programming, volume 3: (2nd ed.) sorting and
  searching}.
\newblock Addison Wesley Longman Publishing Co., Inc., Redwood City, CA, USA,
  1998.

\bibitem[LR98]{LR}
Jean-Louis Loday and Maria Ronco.
\newblock Hopf algebra of the planar binary trees.
\newblock {\em Adv. Math.}, 139(2):293--309, 1998.

\bibitem[Nad11]{nad11}
P.~Nadeau.
\newblock The structure of alternative tableaux.
\newblock {\em J. Combin. Theory Ser. A}, 118(5):1638--1660, 2011.

\bibitem[PWZ96]{PWZ}
M.~Petkov{\v{s}}ek, H.~S. Wilf, and D.~Zeilberger.
\newblock {\em A = B}.
\newblock Ak Peters Series. Peters, 1996.

\bibitem[S{\etalchar{+}}12]{sage}
W.\thinspace{}A. Stein et~al.
\newblock {\em {S}age {M}athematics {S}oftware}.
\newblock The Sage Development Team, 2012.
\newblock {\tt http://www.sagemath.org}.

\bibitem[SCc12]{Sage-Combinat}
The {S}age-{C}ombinat community.
\newblock {S}age-{C}ombinat: enhancing {S}age as a toolbox for computer
  exploration in algebraic combinatorics, 2012.
\newblock {\tt http://combinat.sagemath.org}.

\bibitem[Slo]{oeis}
N.~J.~A. Sloane.
\newblock The {O}n-{L}ine {E}ncyclopedia of {I}nteger {S}equences.
\newblock {\tt http://oeis.org}.

\bibitem[SW07]{steiwil07}
E.~Steingr{\'{\i}}msson and L.~K. Williams.
\newblock Permutation tableaux and permutation patterns.
\newblock {\em J. Combin. Theory Ser. A}, 114(2):211--234, 2007.

\bibitem[Vie07]{vie07}
X.~Viennot.
\newblock Alternative tableaux, permutations and partially asymmetric exclusion
  process.
\newblock {\em Talk at Isaac Newton institute}, April 2007.
\newblock {\tt
  http://www.newton.ac.uk/webseminars/pg+ws/2008/csm/csmw04/0423/viennot/}.

\end{thebibliography}
\label{sec:biblio}

\end{document}